\documentclass[10pt,twocolumn,twoside]{IEEEtran}

\usepackage{latexsym,amssymb,exscale,graphicx,amsfonts,amsmath}
\usepackage{multirow,multicol}
\usepackage{epstopdf}
\usepackage{marvosym}
\usepackage{epsfig} 
\usepackage{setspace}
\usepackage[normalem]{ulem}
\usepackage[table]{xcolor}
\usepackage{xcolor,colortbl}
\usepackage{lipsum}
\usepackage{microtype}
\usepackage{stfloats}
\usepackage[noadjust]{cite}
\usepackage{lettrine}
\usepackage{mathtools, cuted}
\usepackage{amsmath}
\usepackage{blkarray}
\usepackage[lined,ruled,linesnumbered,boxed]{algorithm2e}
\usepackage{mathtools}
\usepackage[english]{babel}
\usepackage[utf8]{inputenc}
\usepackage{bbm}
\usepackage{bm}
\usepackage{amsthm, amssymb}
\newtheorem{theorem}{Theorem}[section] 
\newtheorem{corollary}{Corollary}[section]
\newtheorem{lemma}{Lemma}[section]
\newtheorem{definition}{Definition}[section]
\newtheorem{assumption}{Assumption}
\newtheorem{example}{Example}
\usepackage{comment}

\usepackage{amsbsy}
\def\bftheta{\boldsymbol{\theta}}
\def\bfDelta{\boldsymbol{\Delta}}
\def\bfA{\boldsymbol{A}}
\def\bfB{\boldsymbol{B}}
\def\bfP{\boldsymbol{p}}
\def\bfx{\boldsymbol{x}}
\def\bfy{\boldsymbol{y}}
\def\bfD{\boldsymbol{D}}
\def\bfDelta{\boldsymbol{\Delta}}
\def\supp{\mbox{supp}}

\def\ddefloop#1{\ifx\ddefloop#1\else\ddef{#1}\expandafter\ddefloop\fi}
\def\ddef#1{\expandafter\def\csname bb#1\endcsname{\ensuremath{\mathbb{#1}}}}
\ddefloop ABCDEFGHIJKLMNOPQRSTUVWXYZ\ddefloop
\def\ddef#1{\expandafter\def\csname c#1\endcsname{\ensuremath{\mathcal{#1}}}}
\ddefloop ABCDEFGHIJKLMNOPQRSTUVWXYZ\ddefloop
\def\ddef#1{\expandafter\def\csname v#1\endcsname{\ensuremath{\boldsymbol{#1}}}}
\ddefloop ABCDEFGHIJKLMNOPQRSTUVWXYZabcdefghijklmnopqrstuvwxyz\ddefloop
\def\ddef#1{\expandafter\def\csname v#1\endcsname{\ensuremath{\boldsymbol{\csname #1\endcsname}}}}
\ddefloop {alpha}{beta}{gamma}{delta}{epsilon}{varepsilon}{zeta}{eta}{theta}{vartheta}{iota}{kappa}{lambda}{mu}{nu}{xi}{pi}{varpi}{rho}{varrho}{sigma}{varsigma}{tau}{upsilon}{phi}{varphi}{chi}{psi}{omega}{Gamma}{Delta}{Theta}{Lambda}{Xi}{Pi}{Sigma}{varSigma}{Upsilon}{Phi}{Psi}{Omega}{ell}\ddefloop

\makeatletter
\newcommand{\nosemic}{\renewcommand{\@endalgocfline}{\relax}}
\newcommand{\dosemic}{\renewcommand{\@endalgocfline}{\algocf@endline}}
\let\oldnl\nl
\newcommand{\nonl}{\renewcommand{\nl}{\let\nl\oldnl}}
\makeatother

\newcommand{\ting}[1]{\textcolor{red}{Ting: #1}}
\newcommand{\yudi}[1]{\textcolor{blue}{Yudi: #1}}

\IEEEoverridecommandlockouts

\begin{document}
	
	\title{Power Grid State Estimation under General Cyber-Physical Attacks }
\author{Yudi Huang, \textit{Student Member, IEEE}, Ting He, \textit{Senior Member, IEEE}, Nilanjan Ray Chaudhuri, \\\textit{Senior Member, IEEE}, and Thomas La Porta \textit{Fellow, IEEE}
\thanks{The authors are with the School of Electrical Engineering and Computer Science, Pennsylvania State University, University Park, PA 16802, USA
(e-mail: \{yxh5389, tzh58, nuc88, tfl12\}@psu.edu). }
\thanks{ This work was supported by the National Science Foundation under award ECCS-1836827.}
\thanks{ A preliminary version of this work was presented at SmartGridComm'20~\cite{yudi20SmartGridComm}. }
\vspace{-2em}
}

\maketitle

\begin{abstract}
Effective defense against cyber-physical attacks in power grid requires the capability of accurate damage assessment within the attacked area.
While some solutions have been proposed to recover the phase angles and the link status (i.e., breaker status) within the attacked area, 
existing solutions made the limiting assumption that the grid stays connected after the attack. To fill this gap, we study the problem of recovering the phase angles and the link status under a general cyber-physical attack that may partition the grid into islands. To this end, we (i) show that the existing solutions and recovery conditions still hold if the post-attack power injections in the attacked area are known, and (ii) propose a linear programming-based algorithm that can perfectly recover the link status under certain conditions even if the post-attack power injections are unknown. Our numerical evaluations based on the Polish power grid demonstrate that the proposed algorithm is highly accurate in localizing failed links once the phase angles are known.  \looseness=-1
\end{abstract}
	
\begin{IEEEkeywords}
Power grid state estimation, cyber-physical attack, failure localization.
\end{IEEEkeywords}

\section{Introduction} \label{sec:Intro}

Modern power grids are interdependent cyber-physical systems consisting of a power transmission system (power lines, substations, etc.) and an associated control system (Supervisory Control and Data Acquisition - SCADA and Wide-Area Monitoring Protection and Control - WAMPAC) that monitors and controls the status of the power grid. This interdependency raises a legitimate concern: what happens if an attacker attacks both the physical grid and its control system simultaneously?
The resulting attack, known as a \emph{joint cyber-physical attack}, can cause large-scale blackouts, as the cyber attack can blindfold the control system and thus make the physical attack on the power grid more damaging.  For example, one such attack on Ukraine's power grid left 225,000 people without power for days \cite{Fairley16Spectrum}.

The potential severity of cyber-physical attacks has attracted efforts in countering these attacks \cite{Soltan18TCNS,Soltan17PES}. One of the challenges in dealing with such attacks is that as the cyber attack blocks measurements (e.g., phase angles, breaker status, and so on) from the attacked area, the control center is unable to accurately identify the damage caused by the physical attack (e.g., which lines are disconnected) and hence unable to make accurate mitigation decisions. To address this challenge, solutions have been proposed to estimate the state of the power grid inside the attacked area using power flow models.
Specifically, \cite{Soltan18TCNS} developed methods to estimate the grid state under cyber-physical attacks using the \emph{direct-current (DC) power flow model}, and \cite{Soltan17PES} developed similar methods using the \emph{alternating-current (AC) power flow model}. Both works made the \emph{limiting assumption} that either (i) the grid remains connected after the attack, or (ii) the control center is aware of the supply/demand in each island formed after the attack -- both leading to known post-attack active power injection at each bus.

In practice, however, disconnecting lines within the attacked area may cause partitioning of the grid and change the active power injections, and such changes within the attacked area will not be directly observable to the control center due to the cyber attack. Our goal is thus to estimate the power grid state, especially the breaker status of lines, under cyber-physical attacks without the above assumption.

\subsection{Related Work}


Power grid state estimation, as a key functionality for supervisory control, has been extensively studied in the literature \cite{huang2012state}. Secure state estimation under attack is of particular interest \cite{liu2011false}. Specifically, the attackers can distort sensor data with noise \cite{shoukry2017secure} or inject false data \cite{dan2010stealth} so that the control center cannot correctly estimate the phase angles \cite{vukovic2011network} or the topology \cite{kim2013topology} of the power grid. Recently, joint cyber-physical attacks are gaining attention, as the physical effect of such attacks {is} harder to detect due to the cyber attack \cite{deng2017ccpa, Soltan18TCNS, soltan2018react}.\looseness=-1

In particular, several approaches have been proposed for detecting failed links. In \cite{tate2008line, tate2009double}, the problem is formulated as a mixed-integer program, which becomes computationally inefficient when multiple links fail. The problem is formulated as a sparse recovery problem over an overcomplete representation in \cite{Zhu12TPS, chen2014efficient}, where the combinatorial sparse recovery problem was relaxed to a linear programming (LP) problem.  Based on this approach, the work in \cite{Soltan18TCNS} further 
establishes graph-theoretic conditions 
for accurately recovering the failed links. 
All the algorithms in~\cite{Zhu12TPS, chen2014efficient, Soltan18TCNS} aim to find the sparsest solution among the feasible solutions under the assumption that the power grid remains connected after failure. \looseness=-1

All the state estimation solutions require the modeling of the relationship between the observable parameters and the unknown variables of interest. To this end, two types of models have been considered: DC power flow model and AC power flow model. The AC power flow model~\cite{stevenson1994power} is based on the AC power flow equations, which can represent the voltage magnitude and phase angle at each bus in the system. The DC power flow model~\cite{Stott09TPS} is an approximation of the AC power flow model by neglecting the resistive losses and assuming a uniform voltage profile. In the literature of state estimation and particularly failure localization, most existing solutions are based on the DC power flow model~\cite{Zhu12TPS,chen2014efficient,Soltan18TCNS,tate2008line,tate2009double}, with few exceptions~\cite{Garcia16TPS,Soltan17PES}. We adopt the DC power flow model in this work due to its simplicity and robustness, and leave extensions to the AC power flow model to future work.  

\subsection{Summary of Contributions}

We aim at estimating the power grid state within an attacked area, focusing on the phase angles and the link status (i.e., breaker status of lines), 
with the following contributions: \begin{enumerate}
    \item We show that an existing rank-based condition for recovering the phase angles, previously established when the grid remains connected after the attack,  still holds without this limiting assumption.
    \item We show that existing graph-theoretical conditions for localizing the failed links, previously established under the above assumption of a connected grid, still hold without this assumption if the post-attack power injections are known.
    \item When the post-attack power injections are unknown but the phase angles are known, we develop an LP-based algorithm that is guaranteed to correctly identify the status of failed/operational links under certain conditions. \looseness=-1
    \item Our evaluations on a large grid topology show that the proposed algorithm is highly accurate in localizing the failed links with few false alarms, while the rank-based condition for recovering the phase angles can be hard to satisfy, signaling the importance of protecting PMU measurements. 
\end{enumerate}

\textbf{Roadmap.} Section~\ref{sec:Problem Formulation} formulates our overall problem, which is divided into three subproblems addressed in Sections~\ref{sec:Recovery of Phase Angles}--\ref{sec:Localizing Failed Links and Recover Active Powers}. Then Section~\ref{sec:Performance Evaluation} evaluates our solutions on a real grid topology, and Section~\ref{sec:Conclusion} concludes the paper. {Additional proofs are provided in the appendix. }

\section{Problem Formulation}\label{sec:Problem Formulation}

\subsection{Power Grid Model}

We model the power grid as a connected undirected graph $G=(V,E)$, where $V$ is the set of nodes (buses) and $E$ the set of links (transmission lines). Each link $e=(s,t)$ is associated with a \emph{reactance} $r_{st}$ ($r_{st} = r_{ts}$) and a status $\in \{\mbox{operational}, \mbox{failed}\}$ (assumed to be operational before attack). Each node $v$ is associated with a phase angle $\theta_v$ and an active power injection $p_v$. The phase angles $\bftheta:=(\theta_v)_{v\in V}$ and the active powers $\bfP:=(p_v)_{v\in V}$ are related by
\begin{align}\label{eq:B theta = p}
    \bfB \bftheta = \bfP,
\end{align}
where $\bfB:=(b_{uv})_{u,v\in V} \in \mathbb{R}^{|V|\times|V|}$ is the \emph{admittance matrix}, defined as:
\begin{align}
    b_{uv} &=\left\{\begin{array}{ll}
    0 & \mbox{if }u\neq v, (u,v)\not\in E,\\
    -1/r_{uv} &  \mbox{if } u\neq v, (u,v)\in E,\\
    -\sum_{w\in V\setminus \{u\}}b_{uw} &\mbox{if }u=v.
    \end{array}\right.
\end{align}
Given an arbitrary orientation of each link, the topology of $G$ can also be represented by the \emph{incidence matrix} $\bfD\in \{-1,0,1\}^{|V|\times|E|}$, whose $(i,j)$-th entry is defined as
\begin{align}
    d_{ij} &= \left\{\begin{array}{ll}
    1 & \mbox{if link }e_j\mbox{ comes out of node }v_i,\\
    -1 & \mbox{if link }e_j\mbox{ goes into node }v_i,\\
    0 & \mbox{otherwise.}
    \end{array}\right.
\end{align}

We assume that each node is deployed with a phasor measurement unit (PMU) measuring the phase angle and remote terminal units (RTUs) measuring the active power injection at this node, and the status and the power flows of its incident links. These reports are sent to the control center via a SCADA or WAMPAC system. The PMU data is assumed to be communicated over a relatively secure WAMPAC network, and the RTU measurements over a more vulnerable SCADA network. \looseness=-1 

{\emph{Remark:} 
Under the North American SynchroPhasor Initiative (NASPI)~\cite{dagle2010north}, the number of PMUs is steadily growing~\cite{PMUdeployment}, and installing PMUs is becoming part of routine transmission system upgrades and new construction~\cite{NASPI2015Kit}. Some utilities have achieved full observability in their networks, e.g., Dominion Power has piloted the PMU-based linear state estimator~\cite{jones2013three,jones2014methodology}. These trends indicate that it is just a matter of time that complete observability through PMUs is achieved, which is the scenario we will consider.\looseness=-1
%
}

\subsection{Attack Model}


\begin{figure}[tb]
\centering
\includegraphics[width=.6\linewidth]{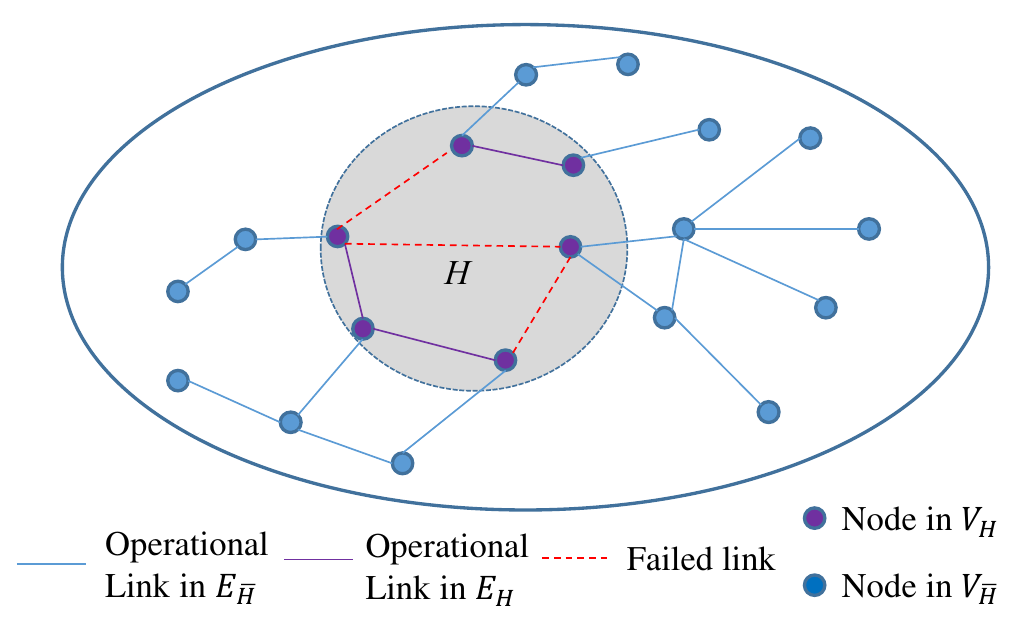}
\vspace{-1em}
\caption{A cyber-physical attack that blocks information from the attacked area $H$ while disconnecting certain lines within $H$. 
} \label{fig:cyber_physical_attack_hyd}
\vspace{-.5em}
\end{figure}

As illustrated in Fig.~\ref{fig:cyber_physical_attack_hyd}, an adversary attacks an area $H$ of the power grid by: (i) blocking reports from the nodes within $H$ (cyber attack), and (ii) disconnecting a set $F$ ($|F|>0$) of links within $H$ (physical attack). Formally, $H=(V_H,E_H)$ is a subgraph induced by a set of nodes $V_H$, where $E_H$ is the set of links for which both endpoints are in $V_H$. 

\subsection{State Estimation Problem}

\begin{table}[tb]
\footnotesize
\renewcommand{\arraystretch}{1.3}
\caption{Notations} \label{tab:notation}
\vspace{-.5em}
\centering
\begin{tabular}{c|l}
  \hline
  Notation & Description  \\
  \hline
 $G=(V,E)$ & power grid \\
 \hline
 $H$, $\bar{H}$ & attacked/unattacked area \\
  \hline
 $F$ {, $E_1$} & set of failed{/operational} links {after attack} \\
 \hline
 $\bfB${, $\bfD$} & admittance{/incidence} matrix \\
 \hline
 $\bftheta${, $\bftheta'$} & vector of phase angles {before/after attack} \\
 \hline
 $\bfP${, $\bfP'$} & vector of active power injections {before/after attack} \\
 \hline
 $\bm{\Gamma}$ & $\text{diag}\{\frac{1}{r_e}\}_{e\in E}$ ($r_e$: reactance of link $e$)\\
 \hline
 $\bfDelta$ & vector of changes in active power injections \\
  \hline
 {$\tilde{\vD}$} &  {matrix of (hypothetical) post-attack power flows \eqref{eq:tilde{D}} }\\
 \hline
 {$\eta$} & {rounding threshold in Algorithm~\ref{alg: fedgeDet}} \\
 \hline
 {$S_U, f_{U,g}, f_{U,1(0)}$}   & {definitions related to hyper-node \eqref{eq:properties of hyper-node}}
 \\
 \hline
 {$S_T, f_{T,g}, f_{T,1(0)}$ } & {definitions related to fail-cover hyper-node \eqref{eq:def about fail-cover hyper-node}}
 \\
  \hline
\end{tabular}
\vspace{-.5em}
\end{table}
\normalsize

\textbf{Notation.} The main notations are summarized in Table~\ref{tab:notation}.
Moreover, given a subgraph $X$ of $G$, $V_X$ and $E_X$ denote the subsets of nodes/links in $X$, and $\bfx_X$ denotes the subvector of a vector $\bfx$ containing elements corresponding to $X$. Similarly, given two subgraphs $X$ and $Y$ of ${G}$, $\bfA_{X|Y}$ denotes the submatrix of a matrix $\bfA$ containing rows corresponding to $X$ and columns corresponding to $Y$.
We use $D_H\in \{-1,0,1\}^{|V_H|\times|E_H|}$ to denote the incidence matrix of the  attacked area $H$.
For each quantity $x$, we use $x'$ to denote its value after the attack. {For a set $A$, $\mathbb{I}_{A}(e) = 1$ if $e\in A$ and $\mathbb{I}_{A}(e) = 0$ otherwise.}

\textbf{Goal.}
Our goal is to recover the post-attack phase angles $\bftheta'_H$ and localize the failed links $F$ within the attacked area, based on the state variables before the attack and the measurements from the unattacked area $\bar{H}$ after the attack.

In contrast to the previous works, we consider cases where the attack may partition the grid into multiple islands, which can cause changes in active power injections to maintain the supply/demand balance in each island.
Let $\bfDelta = (\Delta_v)_{v\in V} := \bfP-\bfP'$ denote the change in active power injections, where $\Delta_v>0$ if $v$ is a generator bus and $\Delta_v\le0$ if $v$ is a load bus. 


\section{Recovery of Phase Angles}\label{sec:Recovery of Phase Angles}

Under the assumption that $G$ remains connected after the attack and thus $\bfDelta = \bm{0}$, \cite{Soltan18TCNS} showed that the post-attack phase angles $\bftheta'_H$ can be recovered if the submatrix $\bfB_{\bar{H}|H}$ of the admittance matrix has a full column rank. Below, we will show that the same condition holds without this limiting assumption. 

Specifically, we have the following lemma ({see proof in appendix}) that extends \cite[Lemma~1]{Soltan18TCNS} to the case of arbitrary $\bfDelta$ (``$\supp$'': indices of non-zero entries in the input vector).\looseness=-1

\begin{lemma}\label{lem:supp in V_H}
$\supp(\bfB(\bftheta-\bftheta')-\bfDelta) \subseteq V_H$.
\end{lemma}


Using Lemma~\ref{lem:supp in V_H},
we prove that the recovery condition in \cite[Theorem~1]{Soltan18TCNS} remains sufficient even if the post-attack grid may be disconnected ({see proof in appendix}). \looseness=-1

\begin{theorem}\label{thm:recovery of phase angles}
The phase angles $\bftheta'_H$ within the attacked area can be recovered correctly if $\bfB_{\bar{H}|H}$ has a full column rank.
\end{theorem}



\section{Localizing Failed Links with Known Active Powers}\label{sec:Localizing Failed Links}

Now assume that the post-attack phase angles $\bftheta'$ have been recovered. This can be achieved when $\bfB_{\bar{H}|H}$ has a full column rank {as shown in Theorem~\ref{thm:recovery of phase angles}}. Alternatively, $\bftheta'$ can be directly reported from PMUs -- assuming that the cyber attack only affects SCADA but not the WAMPAC network carrying PMU measurements. {This can occur in a hybrid control system, where the PMU measurements are reported via a modern WAMPAC network and the other sensor measurements are reported via a legacy SCADA network~\cite{WASA}. While SCADA is known to be vulnerable to cyber attacks~\cite{Fairley16Spectrum}, WAMPAC is designed to satisfy stronger cyber security requirements~\cite{WAMPACsecurity}.}
We will show that as long as the change in active powers $\bfDelta$ is known, the failed links can be uniquely localized under the same conditions as specified in \cite{Soltan18TCNS}.  \looseness=-1

First, we note that under practical assumptions, the conditions presented in Section~\ref{sec:Recovery of Phase Angles} for recovering the phase angles greatly simplify the recovery of the active powers. To this end, we assume that the adjustment of active power injections at generator/load buses follows the \emph{proportional load shedding/generation reduction policy}, where (i) either the load or the generation (but not both) will be reduced upon the formation of an island, and (ii) if nodes $u$ and $v$ are in the same island and of the same type (both load or generator), then $p_u'/p_u = p_v'/p_v$. This policy models the common practice in adjusting load/generation in the case of islanding~\cite{pal2006robust,lu2016under}.
{It implies that (1) the machine rating of a generator is proportional to its pre-disturbance (nominal) power output, as is typical in power grid operation, and (2) in the case of demand more than supply in an island, the frequency nadir during the inertial phase becomes low enough to activate underfrequency relays, and the load shedding action that follows leads to balance of supply and demand before the governor response kicks in (due to slow turbine-governor time constants involved in that process). }\looseness=-1

We observe the following cases 
in which the active powers can be recovered via this policy ({see proof in appendix}). 

\begin{lemma}\label{lem:recover delta}
Let $N(v;\bar{H})$ denote the set of all the nodes in $\bar{H}$ that are connected to node $v$ via links in $E\setminus E_H$.
Then under the proportional load shedding policy, $\Delta_v$ for $v\in V_H$ can be  recovered unless $N(v;\bar{H})=\emptyset$ or every $u\in N(v;\bar{H})$ is of a different type from $v$ with $\Delta_u=0$.
\end{lemma}

\emph{Remark~1:} Under the condition of Theorem~\ref{thm:recovery of phase angles}, i.e., $\bfB_{\bar{H}|H}$ has a full column rank, each $v\in V_H$ must be the neighbor of at least one node in $\bar{H}$ (otherwise its corresponding column in $\bfB_{\bar{H}|H}$ will be $\mathbf{0}$), and thus $N(v;\bar{H})\neq \emptyset$. Moreover, majority of the nodes in practice are load buses, and thus each node in $H$ is likely to be a load bus neighboring to another load bus in $\bar{H}$.
Thus, we can usually recover $\bfDelta_H$ under the proportional load shedding policy if the condition for recovering $\bftheta'_H$ holds.  

\emph{Remark~2:} Besides the cases indicated in Lemma~\ref{lem:recover delta}, it is also easy to show that if $H$ contains no generator bus or no load bus, and $\sum_{v\in V_{\bar{H}}} \Delta_v = 0$, then we must have $\bfDelta_H=\bm{0}$.

Next, we will establish the conditions for localizing the failed links $F$ with known $\bftheta'$ and $\bfDelta$. The basic observation is the following property of the set $F$ ({see proof in appendix}). 

\begin{lemma}\label{lem:existence of solution with supp = F}
There exists a vector $\bfx\in \mathbb{R}^{|E_H| }$ that satisfies $\supp(\bfx) = F$, and
\begin{align}\label{eq:D_H x equation}
\bfD_H \bfx = \bfB_{H|G}(\bftheta - \bftheta') - \bfDelta_H.
\end{align}
\end{lemma}

This lemma, which replaces \cite[Lemma~2]{Soltan18TCNS}, implies that if one can find the conditions under which the solution to (\ref{eq:D_H x equation}) is unique, then the links corresponding to non-zero elements of this solution must be the failed links.
To this end, \cite{Soltan18TCNS} gave a set of graph-theoretic conditions. As these conditions are only about the solution space of $\bfD_H \bfx = \bfy,
$
they remain valid in our setting as long as the righthand side is known. 
We summarize these conditions below ({see proof in appendix}). \looseness=0 

\begin{theorem}\label{thm:localize failed links}
The failed links $F$ within the attacked area can be localized correctly if:\begin{enumerate}
\item $H$ is acyclic (i.e., a tree or a set of trees), in which case (\ref{eq:D_H x equation}) has a unique solution $\bfx$ for which $\supp(\bfx) = F$, or

\item $H$ is a planar graph satisfying (i) for any cycle $C$ in $H$, $|C\cap F|<|C\setminus F|$, and (ii) $F^*$ is $H^*$-separable\footnote{Here $H^*$ is the dual graph of $H$, and $F^*$ is the set of edges in $H^*$ such that each edge in $F^*$ connects a pair of vertices that correspond to adjacent faces in $H$ separated by a failed link.}, in which case the optimization $\min \|\bfx\|_1$ s.t. (\ref{eq:D_H x equation}) has a unique solution $\bfx$ for which $\supp(\bfx) = F$.
\end{enumerate}
\end{theorem}

Special cases satisfying the second condition in Theorem~\ref{thm:localize failed links} include that (i) $H$ is a cycle in which majority of the links have not failed, and (ii) $H$ is a planar bipartite graph in which each cycle contains fewer failed links than non-failed links \cite{Soltan18TCNS}.

\section{Localizing Failed Links with Unknown Active Powers}\label{sec:Localizing Failed Links and Recover Active Powers}

Although providing strong theoretical guarantees, the solutions for localizing failed links given in Section~\ref{sec:Localizing Failed Links} are only applicable to small attacked areas with simple topologies (e.g., trees or cycles in which every node is connected to another node outside the attacked area). To deal with larger attacked areas for which $\bfDelta_H$ cannot be recovered by Lemma~\ref{lem:recover delta}, we investigate alternative solutions by jointly estimating the set of failed links $F$ and the changes in active power injections $\bfDelta_H$.
As in Section~\ref{sec:Localizing Failed Links}, we assume that the post-attack phase angles $\bftheta'$ are known, which can be either inferred or directly measured.\looseness=-1 


\subsection{Solution}

Our approach is to formulate the joint estimation problem as an optimization as follows.

\emph{Constraints:} Let $\bfx \in \{0,1\}^{|E| }$ be an indicator vector such that $x_e=1$ if and only if $e \in F$. Due to $\bfB = \bfD\bm{\Gamma}\bfD^T$ (see Table~\ref{tab:notation} for the definitions), we can write the post-attack admittance matrix as $\bfB' = \bfB - \bfD\bm{\Gamma}\text{diag}\{\bfx\}\bfD^T$, which implies \looseness=-1
\begin{align}\label{eq:pf_constraint}
    \bfDelta_H = \bfB_{H|G}(\bftheta-\bftheta') + \bfD_{H}\bm{\Gamma}_H\text{diag}\{\bfD_{G|H}^T \bftheta'\}\bm{x}_H,
\end{align}
where $\bfD_{G|H}\in \{-1,0,1\}^{|V|\times |E_H|}$ is the submatrix of the incidence matrix $\bfD$ only containing the columns corresponding to links in $H$. For simplicity, we define
\begin{align}\label{eq:tilde{D}}
\tilde{\bfD} := \bfD \bm{\Gamma} \text{diag}\{\bfD^T \bftheta'\}.
\end{align}
For link $e_k=(i,j)$, $(\bm{\tilde{D}})_{i,k} = -(\bm{\tilde{D}})_{j,k} = \frac{\theta_i'-\theta_j'}{r_{ij}}$, which indicates the post-attack power flow on link $e_k$ if it is operational.

Besides \eqref{eq:pf_constraint}, $\bfDelta_H$ is also constrained as
\begin{subequations}\label{eq:delta_constraint}
\begin{alignat}{2}
&p_v \ge {\Delta_v} \ge 0, &~~\forall v \in \left\{ {u\: |u \in V_H, p_u > 0} \right\}, \label{eq:const_valid_start}&\\
&p_v \le {\Delta_v} \le 0,&\forall v \in \left\{ {u\: |u \in V_H, p_u \le 0} \right\},& \label{eq:const_valid_load}\\
&\boldsymbol{1}^T\bfDelta = 0, \label{eq:const_sum0}&
\end{alignat}
\end{subequations}
which ensures that a generator/load bus will remain of the same type after the attack, and the total power is balanced. It is worth noting that \eqref{eq:const_sum0} is ensured by \eqref{eq:pf_constraint},  which implies that $\bm{1}^T\bfDelta_H-\bm{1}^T\bfB_{H|G}(\bftheta-\bftheta') = (\bm{1}^T\tilde{\bfD}_H)\bm{x}_H = 0$ since $\bm{1}^T\tilde{\bfD}_H = \bm{0}$ by definition \eqref{eq:tilde{D}}. This implies that any $\bfDelta_H$ satisfying \eqref{eq:pf_constraint} will satisfy $\bm{1}^T\bfDelta_H = \bm{1}^T\bfB_{H|G}(\bftheta-\bftheta') = \bm{1}^T\bfDelta^*_H$ ($\bfDelta^*_H$: the ground-truth load shedding values in $H$), and thus satisfy \eqref{eq:const_sum0}. Hence, we will omit \eqref{eq:const_sum0} in the sequel.

\emph{Objective:}
The problem of failure localization aims at finding a set $\hat{F}$ that is as close as possible to the set $F$ of failed links, while satisfying all the constraints. The solution is generally not unique, e.g., if both endpoints of a link $l\in E_H$ are disconnected from $\bar{H}$ after the attack, 
then the status of $l$ will have no impact on any observable variable, and hence cannot be determined.
To resolve this ambiguity, we set our objective as
using the fewest failed links to satisfy all the constraints. This idea has been applied to failure localization in power grid in various forms~\cite{Soltan18TCNS,chen2014efficient}.
Mathematically, the problem is formulated as
\begin{subequations}\label{eq1:L0_binary_original}
\begin{alignat}{2}
(\text{P0}) \quad &\min_{\boldsymbol{x}_H, \bfDelta_H}  \bm{1}^T{\bfx}_H &\\
\mbox{s.t.} \quad
&\eqref{eq:pf_constraint}, \eqref{eq:const_valid_start}-\eqref{eq:const_valid_load},&\\
&x_e\in \{0,1\},~~~\forall e \in E_H,& \label{eq1:P0 xi}
\end{alignat}
\end{subequations}
where the decision variables are $\bfx_H$ and $\bfDelta_H$. 
Next, we characterize the complexity of (P0) ({see proof in appendix}).

\begin{lemma}\label{lem:complexity_p0}
The optimization (P0) is NP-hard. 
\end{lemma}

By relaxing the integer constraint (\ref{eq1:P0 xi}), (P0) is relaxed into
\begin{subequations}\label{eq1:L1_binary_load}
\begin{alignat}{2}
(\text{P1}) \quad &\min_{\boldsymbol{x}_H, \bfDelta_H}  \bm{1}^T{\bfx}_H  & \label{P1: obj}\\
\mbox{s.t.} \quad
&\eqref{eq:pf_constraint}, \eqref{eq:const_valid_start}-\eqref{eq:const_valid_load},&\\
&{\rm{       }}{\bf{0}} \le {{\bfx}_H} \le {\bf{1}}. \label{eq5: const_x_range}&
\end{alignat}
\end{subequations}
where ${\bf{0}} \le \bfx_H \le {\bf{1}}$ denotes element-wise inequality. The problem (P1) is a linear program (LP) which can be solved in polynomial time. Based on (P1), we propose an algorithm for localizing the failed links, given in Algorithm~\ref{alg: fedgeDet}, where the input parameter $\eta\in (0, 1)$ is a threshold for rounding the factional solution of $\bfx_H$ to an integral solution ($\eta=0.5$ in our experiments).
{We will analyze how $\eta$ affects the trade-off between miss rate and false alarm rate of Algorithm~\ref{alg: fedgeDet} at the discussion after Theorem~\ref{lem:no_fa_hyper_direc} in Section~\ref{sec: analysis_sec5}}.

\begin{algorithm}\label{alg: fedgeDet}
\vspace{-.05em}
\SetAlgoLined
\SetKwFunction{Fmain}{FailEdgeDetection}
\SetKwInOut{Input}{input}\SetKwInOut{Output}{output}
\KwIn{$\bfB, \bfP, \bfDelta_{\bar{H}}, \bftheta, \bftheta', \boldsymbol{D}, \eta$ 
}
\KwOut{$\hat{F}$}
Solve the problem (P1) to obtain $\bfx_H$;\\
Return $\hat{F}=\{e:x_e \geq \eta\}$.
\caption{Failed Link Detection }
\end{algorithm}
\vspace{-1.5em}

\subsection{Analysis}\label{sec: analysis_sec5}

We now analyze when the proposed algorithm can correctly localize the failed links. 
In the sequel, $\bfDelta^*_H$ denotes the ground-truth load shedding values in $H$ and $\bfx^*_H$ denotes the ground-truth failure indicators. 

According to \eqref{eq:delta_constraint}, we decompose $V_H$ into $V_{H,L}$ for nodes with $p_v\le 0$ and $V_{H,S}$ for the rest. Define $E_1\subseteq E_H$ as the set of links that operate normally after failure, and $F\subseteq E_H$ as the failed links. We make the following assumption:

\begin{assumption}\label{as:island_allLoad}  As in \cite{Soltan18TCNS}, we assume that for each link $(s,t)\in E_H$, $\theta_s' \ne \theta_t'$, as otherwise the link will carry no power flow and hence its status cannot be identified\footnote{This assumption essentially means that we will ignore the existence of such links in failure localization.}.
\end{assumption}

\subsubsection{Main Results}

First, we simplify (P1) into an equivalent but simpler optimization problem. To this end, we combine the decision variables $\bfDelta_H$ and $\bfx_H$ of (P1) into a single vector $\bfy_H=[\bfDelta_H^T, \bfx_H^T]^T\in \mathbb{R}^{(|E_H|+|V_H|) }$ (where $[A, B]$ denotes horizontal concatenation), and explicitly represent the solution to $\bfy_H$ that satisfies \eqref{eq:pf_constraint}. Notice that \eqref{eq:pf_constraint} can be written as $[\bm{I}_{|V_H|}, -\tilde{\bfD}_H]\bfy_H = \bfB_{H|G}(\bftheta-\bftheta')$ ($\bm{I}_{|V_H|}$: the $|V_H|\times |V_H|$ identity matrix). 
The ground-truth solution $\bfy_H^*=[(\bfDelta_H^*)^T, (\bfx_H^*)^T]^T$ certainly  satisfies \eqref{eq:pf_constraint}. Next, consider the null space of $[\bm{I}_{|V_H|}, -\tilde{\bfD}_H]$, whose dimension is $|E_H|$. It is easy to verify that  $[\tilde{\bm{d}}_{e}^T,\bm{u}^T_e ]^T$ ($e\in E_H$) are $|E_H|$ independent vectors spanning the null space of $[\bm{I}_{|V_H|}, -\tilde{\bfD}_H]$, where $\bm{\tilde{d}}_{e}$ is the column vector of $\bm{\tilde{D}}_H$ corresponding to link $e$, and $\bm{u}_e$ is a unit vector in $\mathbb{R}^{|E_H| }$ with the $e$-th element being $1$ and the other elements being $0$. Therefore, any $\bfy_H$ satisfying \eqref{eq:pf_constraint} can be expressed as \looseness=-1
\begin{align}\label{eq:general_soln}
{{\bfy}_H} = \left[ {\begin{array}{*{20}{c}}
{{\bf{\Delta }}_H^*}\\
{{\bfx}_H^*}
\end{array}} \right] + \sum_{e\in E_H} {{c_e}\left[ {\begin{array}{*{20}{c}}
{{{ \bm{\tilde{d}}_{e}}}}\\
{\bm{u}_e}
\end{array}} \right]},
\end{align}
{where $c_e$'s are the coefficients}. Based on the decomposition of $V_H$ into $V_{H,L}$ and $V_{H,S}$, $\tilde{\bfD}_H$ and $\bfDelta_H$ 
can be written as
\begin{subequations}
\begin{align}\label{eq:tildeD block_both}
\tilde{\bfD}_H &=
\begin{blockarray}{ccc}
{}&{\scriptstyle E_1} & {\scriptstyle F}  \\
\begin{block}{c[cc]}
{\scriptstyle V_{H,L}} & \tilde{\bfD}_{H,L,1} & \tilde{\bfD}_{H,L,F} \\
{\scriptstyle V_{H,S}} & \tilde{\bfD}_{H,S,1} & \tilde{\bfD}_{H,S,F} \\
\end{block}
\end{blockarray},\\
\bfDelta_H &=
\begin{blockarray}{cc}
\begin{block}{c[c]}
 \scriptstyle V_{H,L} & \bfDelta_{H,L}  \\
 \scriptstyle V_{H,S} & \bfDelta_{H,S} \\
\end{block}
\end{blockarray}.
\end{align}
\end{subequations}
Let $\tilde{\bfD}_{H,L}:=[\tilde{\bfD}_{H,L,1}, \tilde{\bfD}_{H,L,F}]$, $\tilde{\bfD}_{H,S}:=[\tilde{\bfD}_{H,S,1}, \tilde{\bfD}_{H,S,F}]$, and $\bm{c}:=(c_e)_{e\in E_H}\in \mathbb{R}^{|E_H| }$. Since $\bfDelta_{H,L}$ and $\bfDelta_{H,S}$ are constrained differently in \eqref{eq:const_valid_start} and \eqref{eq:const_valid_load}, we introduce $\bm{\Lambda}_L = [\bm{I}_{|V_{H,L}|}, \bm{0}]$ and $\bm{\Lambda}_S = [\bm{0}, \bm{I}_{|V_{H,S}|}]$ such that $\bfDelta_{H,L} =\bm{\Lambda}_L\bfDelta_H$, $\bfDelta_{H,S} =\bm{\Lambda}_S\bfDelta_H$, $\tilde{\bfD}_{H,L} = \bm{\Lambda}_L\tilde{\bfD}_{H}$ and $\tilde{\bfD}_{H,S} = \bm{\Lambda}_S\tilde{\bfD}_{H}$. According to \eqref{eq:general_soln}, for Algorithm~\ref{alg: fedgeDet} to correctly localize the failed links, it suffices to have $x^*_e + c_e \ge \eta$ for all $e\in F$ and $x^*_e + c_e < \eta$ for all {$e\in E_1$}. Equivalently, it suffices to ensure that the optimal solution $\bm{c}^*$ to the following optimization problem satisfies ${c}^*_e \ge \eta-1$ for all $e\in F$ and ${c}^*_e < \eta$ for all {$e\in E_1$}: 
\begin{subequations}\label{eq:primal_theo_both}
\begin{alignat}{2}
\min_{\boldsymbol{c}} \quad  &\bm{1}^T \bm{c}  \label{min_c: obj} \\
\mbox{s.t.} \quad
&\bm{\tilde{D}}_{H,L}\bm{c} \le -\bfDelta_{H,L}^*, \label{eq2:c_const_start}  \\
&-\bm{\tilde{D}}_{H,L}\bm{c} \le -(\bm{\Lambda}_L\bfP_H-\bfDelta_{H,L}^*), \label{eq2:c_const_p} \\
&-\bm{\tilde{D}}_{H,S}\bm{c} \le \bfDelta_{H,S}^*, \label{eq2:c_const_g0} \\
&\bm{\tilde{D}}_{H,S}\bm{c} \le \bm{\Lambda}_S\bfP_H-\bfDelta_{H,S}^*,\label{eq2:c_const_gP} \\
&-\bm{c} \le {\bfx}_H^*, \label{eq2:c_const_x0}\\
&\bm{c} \le \bm{1} - {\bfx}_H^*, \label{eq2:c_const_end}
\end{alignat}
\end{subequations}
{where \eqref{min_c: obj} is equivalent to \eqref{P1: obj}, \eqref{eq2:c_const_start}-\eqref{eq2:c_const_p} correspond to \eqref{eq:const_valid_load}, \eqref{eq2:c_const_g0}-\eqref{eq2:c_const_gP} correspond to \eqref{eq:const_valid_start}, \eqref{eq2:c_const_x0}-\eqref{eq2:c_const_end} correspond to \eqref{eq5: const_x_range}, and the change of variables $\bm{x}_H, \bfDelta_H$ into $\bm{c}$ based on \eqref{eq:general_soln} ensures the satisfaction of \eqref{eq:pf_constraint}.}
This equivalent formulation of (P1) will help to simplify our analysis by eliminating the equality constraint \eqref{eq:pf_constraint}. \emph{For notational simplicity, we will omit the subscript $_H$ 
in the sequel} unless it causes confusion.

Next, we use \eqref{eq:primal_theo_both} to analyze the accuracy of Algorithm~\ref{alg: fedgeDet}. Let $\hat{F}$ be the failed link set returned by Algorithm~\ref{alg: fedgeDet}. We first define $Q_m=F\setminus \hat{F}$ as the set of failed links that are not detected, and $Q_f = \hat{F}\setminus F$ as the set of operational links that are falsely detected as failed. Note that according to \eqref{eq:primal_theo_both}, a failed link $e\in F$ is missed if and only if $c^*_e < \eta-1$. Similarly, an operational link {$e\in E_1$} is falsely detected as failed if and only if $c^*_e \ge \eta$. To express this in a vector form, we define $\bm{W}_m \in \{0,1\}^{|Q_m|\times |E_H|}$ as a binary matrix, where for each $i=1,\ldots,|Q_m|$, {$(W_m)_{i,e} = 1$ if the $i$-th missed link is link $e$}  and thus we have $\bm{W}_m\bm{c}^* \le (\eta-1)\bm{1}$. 
Similarly, $\bm{W}_f \in \{0,1\}^{|Q_f|\times |E_H|}$ is defined such that {$(W_f)_{i,e} = 1$ if the $i$-th false-alarmed link is link $e$}, which leads to $-\bm{W}_f\bm{c}^* \le -\eta\bm{1}$. For ease of presentation, we define
\begin{subequations}\label{eq:inter_sub}
\begin{alignat}{2}
\bm{A}_D^T &:= [\tilde{\bfD}^T_{L}, -\tilde{\bfD}^T_{L}, -\tilde{\bfD}^T_{S}, \tilde{\bfD}^T_{S}] \in \mathbb{R}^{|E_H|\times 2|V_H|}, \\
\bm{A}_x^T &:= [-\bm{I}_{|E_H|}, \bm{I}_{|E_H|}]\in \mathbb{R}^{|E_H|\times 2|E_H|},\\
\bm{W}^T &:= [\bm{W}^T_m, -\bm{W}^T_f]\in \mathbb{R}^{|E_H|\times (|Q_m|+|Q_f|)},\\
\bm{g}_D^T&:=[-(\bfDelta_{L}^*)^T, (-\bfP_{L}')^T, (\bfDelta_{S}^*)^T, (\bfP_{S}')^T],\\
\bm{g}_x^T &:= [(\bfx^*)^T, \bm{1}^T-(\bfx^*)^T]\in \mathbb{R}^{1\times 2|E_H|},\\
\bm{g}_w^T &:= [(\eta-1)\bm{1}^T, -\eta\bm{1}^T]\in \mathbb{R}^{1\times (|Q_m|+|Q_f|)},
\end{alignat}
\end{subequations}
{where $\bfP_L'=\bfP_L-\bfDelta_L^*$ and $\bfP_S'=\bfP_S-\bfDelta_S^*$ denote the post-attack active power injections at $V_{H,L}$ and $V_{H,S}$.}
Then the constraints in \eqref{eq:primal_theo_both} can be written as $[\bm{A}_D^T, \bm{A}_x^T]^T \bm{c} \le [\bm{g}_D^T, \bm{g}_x^T]^T$, and the optimal solution must satisfy $\bm{W}\bm{c} \le \bm{g}_w$. 
The following observation is the foundation of our analysis. \looseness=-1

\begin{lemma}\label{lem:ground_alter_gale}
{A link $e\in F$ cannot be missed by Algorithm~\ref{alg: fedgeDet}  if for $Q_m = \{e\}$ and $Q_f = \emptyset$, there is a solution $\bm{z}\ge \bm{0}$ to}
\begin{subequations}\label{eq:alter_gale}
\begin{alignat}{2}
[\bm{A}_D^T, \bm{A}_x^T, \bm{W}^T, \bm{1}]\bm{z} = \bm{0},\label{eq:alter_gale_eq} \\
[\bm{g}_D^T, \bm{g}_x^T, \bm{g}_w^T, \bm{0}]\bm{z} < 0.\label{eq:alter_gale_ineq}
\end{alignat}
\end{subequations}
{Similarly, a link $e'\in E_1$ cannot be falsely detected as failed by Algorithm~\ref{alg: fedgeDet} if there exists a solution $\bm{z}\ge\bm{0}$ to \eqref{eq:alter_gale} where $\vW$ is constructed according to $Q_f = \{e'\}$ and $Q_m = \emptyset$.}

\end{lemma}

{The proof is by contradiction: if $e\in F\setminus \hat{F}$, then for $\bm{W}$ corresponding to $Q_m =\{e\}$ and $Q_f = \emptyset$, there must be no $\bm{z}\geq \bm{0}$ satisfying \eqref{eq:alter_gale}; similar argument holds for $e' \in E_1$ by assuming $e' \in \hat{F}\setminus F$. See detailed proof in Appendix.}\looseness=-1


For ease of presentation, we will introduce a few notations as follows. Denote $\tilde{\bfD}_{u}$ as the row in $\tilde{\bfD}$ corresponding to node $u$, and $\tilde{D}_{u,e}$ as the entry in $\tilde{\bfD}_{u}$ corresponding to link $e$. Recall that as defined in \eqref{eq:tilde{D}}, if $e=(u,v)$, then
 ${\tilde{D}_{u,e} = (\theta'_u - \theta'_v) r_{uv}^{-1}}$.
{We decompose the left-hand-side of \eqref{eq:alter_gale_eq} into $\bm{A}_D^T\vz_{D}+\bm{A}_x^T[\vz_{x-},\vz_{x+}]+\bm{W}_m^T\vz_{w,m}+\bm{W}_f^T\vz_{w,f}+z_*\bm{1}$ such that its row corresponding to link $e$ can be written as
\begin{align}\label{eq5:expand_gale_eq}
\sum_{u\in V_H}\big(\tilde{D}_{u,e} z_{D,u} &- \tilde{D}_{u,e} z_{D,-u}\big) + \big( z_{x+,e}- z_{x-,e}\big) \nonumber \\
&+\mathbb{I}_{Q_m}(e) z_{w,m,e}-\mathbb{I}_{Q_f}(e) z_{w,f,e} + z_*.
\end{align}
Similarly, the left-hand-side of \eqref{eq:alter_gale_ineq} can be expanded into
\begin{align}\label{eq5:expand_gale_ineq}
\sum_{u\in V_H} \big(g_{D,u} z_{D,u} + g_{D,-u} &z_{D,-u}\big) + \sum_{e\in E_H}[ z_{x+,e} (1-x_e^*) + \nonumber \\
&z_{x-,e} x_e^*] +\vg_w^T\vz_w + z_*,
\end{align}
where $\vg_w^T\vz_w=\sum_{e\in E_H}[\mathbb{I}_{Q_m}(e) z_{w,m,e}(\eta-1)-\mathbb{I}_{Q_f}(e) z_{w,f,e}\eta]$, $g_{D,u} := -\Delta_u^*$ and $g_{D,-u} := -p_u'$ if $p_u \le 0$, whereas $g_{D,u} := p_u'$ and $g_{D,-u} := \Delta_u^*$ if $p_u>0$. Then, a solution $\vz\ge \bm{0}$ satisfies \eqref{eq:alter_gale} if $\forall e\in E_H$, we have \eqref{eq5:expand_gale_eq} equal to $0$ and \eqref{eq5:expand_gale_ineq} less than $0$.
}



Although Lemma~\ref{lem:ground_alter_gale} can already be utilized as recovery conditions, it does not explicitly characterize what type of links are guaranteed to be correctly identified. To this end, 
we will show that a link will satisfy the conditions in Lemma~\ref{lem:ground_alter_gale} (and can thus be correctly identified by Algorithm~\ref{alg: fedgeDet}) if its endpoints satisfy certain conditions. To make our conditions as general as possible, we introduce a generalization of node called \emph{hyper-node} as follows (a single node is also a hyper-node):

\begin{definition}\label{def:hyper_node}
A set of nodes $U \subseteq V_H$ is a hyper-node if they induce a connected subgraph before attack. 
\end{definition}

We define a few properties of a hyper-node $U$. Define $E_U$ as the set of links with exactly one endpoint in $U$, i.e, $E_U := \{e|e=(s,t)\in E_H,s\in U, t\notin U\}$. If $E_U\cap F \neq \emptyset$, we define
\begin{subequations}\label{eq:properties of hyper-node}
\begin{alignat}{3}
\tilde{D}_{U,e} &:= \sum_{u\in U}\tilde{D}_{u,e},\\
S_U &:= \{ e\in E_U\setminus F|\: \exists l \in E_U \cap F,  \tilde{D}_{U,l}\tilde{D}_{U,e} > 0\}, \\
f_{U,0} &:=\max_{e\in S_U}|\tilde{D}_{U,e}|\mbox{, where }f_{U,0} := 0 \mbox{ if }S_U = \emptyset,\\
    f_{U,1} &:=\min_{e\in E_U\cap F}|\tilde{D}_{U,e}|,\\
f_{U,g} &:= \hspace{-.25em}
\begin{cases}
{\sum_{u\in U}{g_{D,u}}\mbox{\ \ \ if } {\exists l \in {E_U} \cap F, {\tilde D}_{U,l} < 0} },\\
\sum_{u\in U}{g_{D, - u}} \mbox{\ otherwise.}\\
\end{cases}\label{eq:def_fug}
\end{alignat}
\end{subequations}

\begin{figure}[tb]
\vspace{-.5em}
\centering
\includegraphics[width=.6\linewidth]{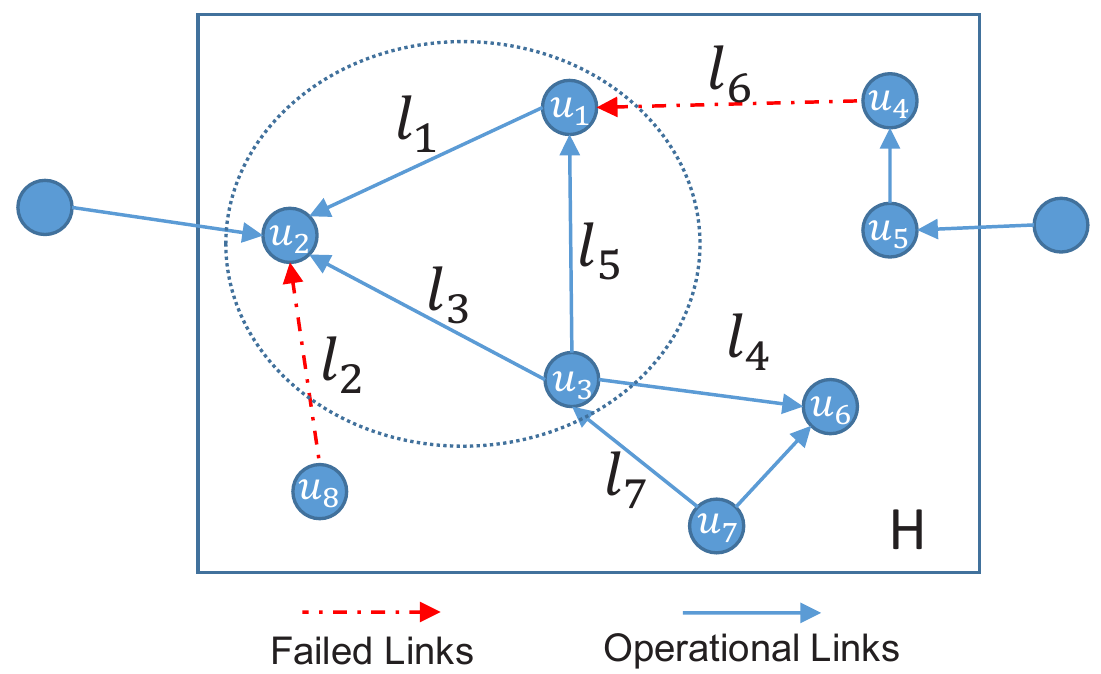}
\vspace{-1em}
\caption{An example of hyper-node (arrow denotes the direction of a power flow or a hypothetical power flow). } \label{fig:eg_hyper_node}
\vspace{-1em}
\end{figure}

\begin{example}
Consider an attacked area $H$ as shown in Fig.~\ref{fig:eg_hyper_node}, where blue circles denote nodes (buses) while the direction of each link indicates the direction of power flow\footnote{These may be hypothetical power flows, as a failed link {carries no flow}.}. Suppose that $F = \{l_2, l_6\}$ and all nodes are load buses. Nodes $u_1, u_2$ and $u_3$ form a hyper-node $U$, where $E_U = \{l_2,l_4,l_6,l_7\}$, $S_U = \{l_7\}$, $f_{U,0} = |\tilde{D}_{U,l_7}|$, $f_{U,1} = \min\{ |\tilde{D}_{U,l_2}|, |\tilde{D}_{U,l_6}| \} $ and $f_{U,g} = -\sum_{v\in U} \Delta_v^* $. $\tilde{D}_{U,l_1} = \tilde{D}_{u_1,l_1} + \tilde{D}_{u_2,l_1} = 0$ since $l_1\notin E_U$, while $\tilde{D}_{U,l_2} = \tilde{D}_{u_2,l_2} \ne 0$ since $l_2\in E_U$.
\end{example}

Based on these definitions and Lemma~\ref{lem:ground_alter_gale}, we are ready to present a condition under which a failed link $l\in F$ will not be missed by Algorithm~\ref{alg: fedgeDet}.

\begin{theorem}\label{lem:no_Miss_hyper_node}
A failed link $l \in F$ will be detected by Algorithm~\ref{alg: fedgeDet}, i.e., $l\in \hat{F}$, if there exists at least one hyper-node (say $U$) such that $l\in E_U$, for which the following conditions hold:\looseness=-1
\begin{enumerate}
    \item $\forall e, l\in E_U\cap F$, $\tilde{D}_{U,e}\tilde{D}_{U,l} > 0$,
    \item $S_U = \emptyset$, and
    \item $f_{U,g} + (\eta-1)|\tilde{D}_{U,l}|<0$.
\end{enumerate}
\end{theorem}
\begin{proof}
We will prove by showing that there is {a solution to \eqref{eq:alter_gale} for $Q_f = \emptyset$ and $Q_m = \{l\}$ where $l\in E_U$}. We prove this by directly constructing a solution $\bm{z}$ for \eqref{eq:alter_gale} as follows: $\forall u\in U$, if $\tilde{D}_{U,l} < 0$, set $z_{D,u} = 1$; otherwise, set $z_{D,-u} = 1$. Set $z_{w,m,l} = |\tilde{D}_{U,l}|$, $z_{x-,e'} = |\tilde{D}_{U,e'}|$ for $e'\in E_U\setminus F$, $z_{x+,e} = |\tilde{D}_{U,e}|$ for $e\in E_U\cap F \setminus \{l\}$, 
and other entries of $\bm{z}$ to $0$. Note that $(x^*)_{e'} = 0, \forall e'\in E_U \setminus F$, and  $(1-x^*)_{e} = 0, \forall e\in E_U\cap F$. Then, we will demonstrate why \eqref{eq:alter_gale} is satisfied under this assignment of $\bm{z}$. First, {\eqref{eq5:expand_gale_eq} for link $l$} is expanded as $ -|\tilde{D}_{U,l}| + z_{w,m,e} = 0$, and {\eqref{eq5:expand_gale_eq} for $e\in F\setminus \{l\}$} is expanded as $-|\sum_{u\in U}\tilde{D}_{u,e}| + z_{x+,e} = -|\tilde{D}_{U,e}| + z_{x+,e} = 0$ due to condition~1). Second, since $S_U = \emptyset$, for all $e' \in  E_U\setminus F$, the corresponding row in \eqref{eq:alter_gale_eq} is expanded into $|\tilde{D}_{U,e'}| - z_{x-,e'}= 0$. 
Other rows of \eqref{eq:alter_gale_eq} holds trivially since they only involve the zero-entries in the constructed $\bm{z}$. Thus, \eqref{eq:alter_gale_eq} holds under this assignment. As for \eqref{eq:alter_gale_ineq}, its left-hand-side can be expanded as $f_{U,g} + (\eta-1)|\tilde{D}_{U,l}|<0$ due to condition~3). According to Lemma~\ref{lem:ground_alter_gale}, $l\in E_U\cap F$ will not be missed, which completes the proof.
\end{proof}

Based on similar arguments, the following condition can guarantee that an operational link {$l\in E_1$} will not be falsely detected by Algorithm~\ref{alg: fedgeDet} ($l \notin \hat{F}$). For notational simplicity, we first extend the definition of $f_{U,g}$ to a hyper-node $U$ with $E_U\cap F = \emptyset$:
\begin{align}
{f_{U,g}} := \left\{ {\begin{array}{*{20}{ll}}
{\sum_{u\in U}{g_{D,u}}\mbox{\ \ \ if } {\exists l \in {E_U} \setminus F, {\tilde D}_{U,l} > 0} },\\
\sum_{u\in U}{g_{D, - u}} \mbox{\ otherwise.} 
\end{array}} \right.
\end{align}

\begin{theorem}\label{lem:no_fa_hyper_direc}
An operational link {$l\in E_1$} will not be detected (as failed) by Algorithm~\ref{alg: fedgeDet}, i.e., $l\notin \hat{F}$, if there exists at least one hyper-node (say $U$) such that $l\in E_U$, for which the following conditions hold:\looseness=0
\begin{enumerate}
    \item 
    {$\forall l, l'\in E_U\cap E_1:$} $ \tilde{D}_{U,l}\tilde{D}_{U,l'} > 0$,
    \item $S_U = \emptyset$ if $E_U\cap F \ne \emptyset$, and
    \item $f_{U,g}-\eta|\tilde{D}_{U,l}|<0$.
\end{enumerate}
\end{theorem}
\begin{proof}
Similar to the proof of Theorem~\ref{lem:no_Miss_hyper_node}, we will prove by showing that there is {a solution to \eqref{eq:alter_gale} for $Q_f = \{l\}$ and $Q_m = \emptyset$ where $l\in E_U$}. We construct the following $\bm{z}$: $\forall u\in U$, if $\tilde{D}_{U,l} < 0$, set $z_{D,u} = 1$; otherwise, set $z_{D,-u} = 1$. Set $z_{w,f,l} = |\tilde{D}_{U,l}|$, $z_{x-,e'} = |\tilde{D}_{U,e'}|$ for $e'\in E_U\setminus (F\cup \{l\})$, $z_{x+,e} = |\tilde{D}_{U,e}|$ for $e\in E_U\cap F$, and other entries of $\bm{z}$ to 0. Then, it is easy to check that \eqref{eq:alter_gale_eq} is satisfied. As for \eqref{eq:alter_gale_ineq}, considering that  $\bm{g}_x^T\bm{z}_x = \sum_{e'\in E_U\setminus F}x_{e'}^* + \sum_{e\in E_U\cap F}(1-x_{e}^*)= 0$ since $(x^*)_{e'} = 0,\: \forall e'\in E_U \setminus F$ and $(1-x^*)_{e} = 0, \: \forall e\in E_U\cap F$, the left-hand-side of \eqref{eq:alter_gale_ineq} can be expanded as
\begin{align}
\bm{g}_D^T\bm{z}_D + \bm{g}_x^T\bm{z}_x -\eta z_{w,f,l} = f_{U,g}-\eta |\tilde{D}_{u,l}| < 0,
\end{align}
where $f_{U,g} = \sum_{u\in U} g_{D,u}$ if $\tilde{D}_{U,l} > 0$ and $f_{U,g} = \sum_{u\in U} g_{D,-u}$ if $\tilde{D}_{U,l} < 0$, and the last inequality holds due to condition 3). Thus, according to Lemma~\ref{lem:ground_alter_gale}, $l \notin \hat{F}$, which completes the proof.
\end{proof}

\emph{Remark:} 
Theorems~\ref{lem:no_Miss_hyper_node} and \ref{lem:no_fa_hyper_direc} provide sufficient conditions for Algorithm~\ref{alg: fedgeDet} to correctly identify the status of a link $l$ based on the direction and magnitude of power flows around a hyper-node $U$ at the ``endpoint'' of $l$ (i.e., $l\in E_U$): 
\begin{enumerate}
    \item {The (hypothetical) power flows on all the links of the same status (failed or operational) around $U$ should be in the same direction, i.e., all going into or out of $U$ (condition~1); all links of the different status (if any) around $U$ should have opposite (hypothetical) power flow directions (condition~2); the magnitude of the (hypothetical) power flow on the link of interest (i.e., $l$) should be sufficiently large (condition~3).}
    \item {The value of $\eta$ can be tuned to control the trade-off between miss rate and false alarm rate. Specifically, the condition~3) of Theorem~\ref{lem:no_Miss_hyper_node} will be easier to satisfy with a smaller $\eta$. On the contrary, the condition~3) of Theorem~\ref{lem:no_fa_hyper_direc} will be easier to satisfy with a larger $\eta$. 
    } \looseness=-1
\end{enumerate}

However, the conditions in Theorems~\ref{lem:no_Miss_hyper_node} and \ref{lem:no_fa_hyper_direc} are stronger than necessary. 
In particular, $z_*$ is always $0$ in the proof of these theorems, which means that the optimality condition $\bm{1}^T\bm{c}^*\le 0$ formulated in \eqref{eq:feasibility_incorrect} has not been exploited. This results in the requirement of the condition ``$S_U=\emptyset$'' in these theorems.
To better characterize the accuracy of Algorithm~\ref{alg: fedgeDet}, we will establish a condition 
that exploits the optimality condition. To this end, we introduce a few further definitions as follows.

\begin{definition}
A hyper-node $U$ is a \emph{fail-cover hyper-node} if $E_U\cap F \ne \emptyset$ and $\forall e, l\in E_{U}\cap F$, $\tilde{D}_{U,e}\tilde{D}_{U,l} > 0$.
\end{definition}

Given a set of fail-cover hyper-nodes $T$, we divide $T$ into $T_n = \{U_{n_i}\}$ and $T_p = \{U_{p_i}\}$, such that each $U_{n_i} \in T_n$ satisfies $\tilde{D}_{U_{n_i},e} \leq 0$ for all $e\in F$, and each $U_{p_i} \in T_p$ satisfies $\tilde{D}_{U_{p_i},e} \ge 0$ for all $e\in F$. Then, 
we define:
\begin{subequations}\label{eq:def about fail-cover hyper-node}
\begin{align}
R_{U_i} &:= \frac{\max_{U\in T}\{ f_{U,1} \}}{f_{U_i,1}}, ~~~\forall U_i\in T, \\
\tilde{D}_{T,e} &:= \sum_{U_i\in T_n} R_{U_i}\tilde{D}_{U_i,e} + \sum_{U_i\in T_p} (-R_{U_i})\tilde{D}_{U_i,e}, \label{eq:def_DTe}\\
S_T &:= \{e'\notin F| \exists e\in F \mbox{ s.t. } \tilde{D}_{T,e'}\tilde{D}_{T,e} > 0 \},\\
f_{T,0} &:= \max_{e'\in S_T} |\tilde{D}_{T,e'}|,~~~f_{T,1} := \min_{e\in F}|\tilde{D}_{T,e}|, \\
f_{T,g} &:= \sum_{U_i\in T} R_{U_i} f_{U_i,g}.
\end{align}
\end{subequations}

Now we provide the recovery conditions when condition~2) in Theorems~\ref{lem:no_Miss_hyper_node} and \ref{lem:no_fa_hyper_direc} is relaxed ({see proof in appendix}).

\begin{theorem}\label{lem: final_lemma}
Assume that there exists a set of fail-cover hyper-nodes $T = \{U_i\}$ satisfying the following conditions:
\begin{enumerate}
    \item $\forall e \in F, |\tilde{D}_{T,e}| \ge f_{T,0}$, i.e., $f_{T,1} \ge f_{T,0}$, and
    \item $F \subseteq \bigcup_{U_i\in T} E_{U_i}$.
\end{enumerate}
Then, a failed link $e\in F$ will be detected by Algorithm~\ref{alg: fedgeDet} ($e \in \hat{F}$) if
\begin{align}\label{eq: no_miss_hyper_f}
f_{T,g} + (\eta-1)(|\tilde{D}_{T,e}| - f_{T,0}) < 0.
\end{align}
In addition, an operational link {$e'\in E_1$} will not be detected (as failed) by Algorithm~\ref{alg: fedgeDet}  ($e'\notin \hat{F}$) if
\begin{align}\label{eq: no_fa_hyper_f}
\left\{ {\begin{array}{*{20}{c}}
{{f_{T,g}} - \eta \left( {f_{T,1}} + {\left| {\tilde D_{{T},e'}} \right| } \right) < 0} \mbox{ if } e' \notin S_{T}\\
{{f_{T,g}} - \eta \left( {f_{T,1}} - {\left| {\tilde D_{{T},e'}} \right| } \right) < 0} \mbox{ if } e' \in S_{T}
\end{array}} \right.
\end{align}
\end{theorem}

{Besides serving as performance guarantees for Algorithm~\ref{alg: fedgeDet}, the conditions in Lemma~\ref{lem:ground_alter_gale} and Theorems~\ref{lem:no_Miss_hyper_node}--\ref{lem: final_lemma} can be used in real-time contingency analysis, where the operator can simulate a certain number of failures and test their detectability based on these conditions.  Moreover, these conditions can serve as the foundation for stronger conditions that can be tested in the field to verify the correctness of the failure localization results before crew dispatching~\cite{verifiable21arXiv}.} \looseness=-1

For ease of understanding, we visualize the relationship among the results in this section in Fig.~\ref{fig:relationship_theorems}. Specifically,
\begin{itemize}
    \item {Even if the condition in Lemma~\ref{lem:ground_alter_gale} is not satisfied for a link, Algorithm~\ref{alg: fedgeDet} may still identify its state  correctly, 
    because (P1) can have multiple optimal solutions while a typical LP solver can only return one of them, which can result in correct identification of the given link.  Thus, the cases satisfying Lemma~\ref{lem:ground_alter_gale} are only part of all the cases in which Algorithm~\ref{alg: fedgeDet} is correct.}
    \item The cases satisfying Lemma~\ref{lem:ground_alter_gale} can be classified into four categories according to whether they satisfy Theorem~\ref{lem:no_Miss_hyper_node}/\ref{lem:no_fa_hyper_direc} and/or Theorem~\ref{lem: final_lemma}.
    \item At a first glance, \eqref{eq: no_miss_hyper_f} in Theorem~\ref{lem: final_lemma} seems to include all the cases covered by Theorem~\ref{lem:no_Miss_hyper_node} due to the relaxed condition on $S_U$. However, \eqref{eq: no_miss_hyper_f} depends on the value of $f_{T,g}$, which is the weighted sum of all $f_{U_i,g}$ for $U_i\in T$. For a failed link $e$, it is possible to find a hyper-node $U$ with $f_{U,g}= 0$ such that Theorem~\ref{lem:no_Miss_hyper_node} holds, while \eqref{eq: no_miss_hyper_f} is violated. This argument applies similarly to a link $e'\in E_H\setminus F$. Thus, cases covered by Theorem~\ref{lem:no_Miss_hyper_node}/\ref{lem:no_fa_hyper_direc} and those covered by Theorem~\ref{lem: final_lemma} partially overlap.  A quantitative analysis will be shown in Fig.~\ref{fig:relationship}. 
\end{itemize}

\begin{figure}[tb]
\vspace{-.0em}
\centering
\includegraphics[width=.5\linewidth]{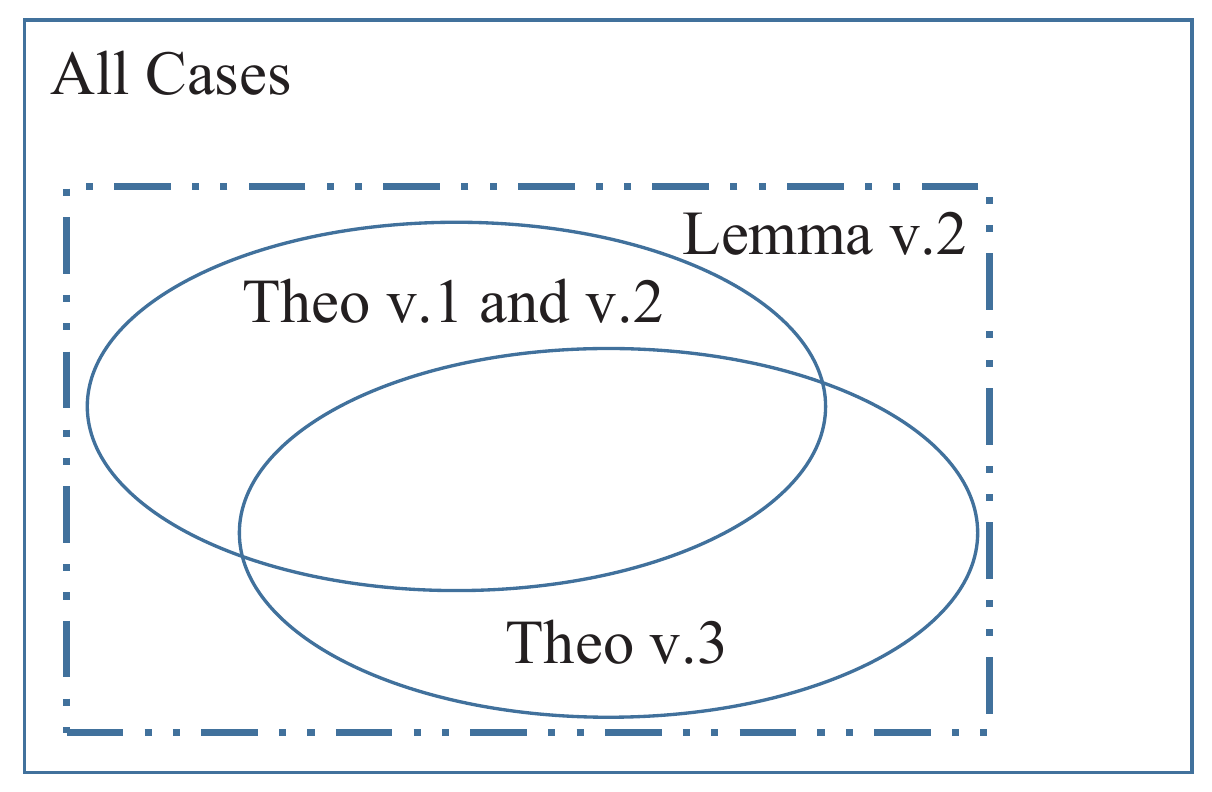}
\vspace{-.5em}
\caption{Relationship between cases covered by various theorems. } 
\label{fig:relationship_theorems}
\vspace{-1em}
\end{figure}

{Theorems~\ref{lem:no_Miss_hyper_node}--\ref{lem: final_lemma} shed light on what kind of $H$ allows accurate failure localization. For example, Theorem~\ref{lem: final_lemma} implies that Algorithm~\ref{alg: fedgeDet} will localize all the failures correctly if $\exists$ a vertex cover $V_0$ of $H$ (i.e., $\forall e\in E_H$ is incident to some $u\in V_0$), where $\forall u \in V_0$ satisfies
(i) $p_{u} = 0$, 
(ii) $f_{\{u\},1} > f_{\{u\},0}$, 
and (iii) $u$ is incident to at most one failed link. \looseness=-1
}

\begin{example}
Consider Fig.~\ref{fig:eg_hyper_node} as an example, where all the nodes in $H$ are load buses. If $\Delta_u^* = 0, \forall u\in V_2:= \{u_1,u_2,u_3,u_6,u_7\}$, then $T = \{V_2\}$ satisfies that $f_{T,g} = \sum_{u\in V_2}g_{D,u} = -\sum_{u\in V_2}\Delta_{u}^* = 0$ and $S_{T} = \emptyset$, which leads to the satisfaction of \eqref{eq: no_miss_hyper_f} for all the failed links and \eqref{eq: no_fa_hyper_f} for all the operational links. Meanwhile, $l_2$ and $l_6$ are guaranteed to be correctly identified through Theorem~\ref{lem:no_Miss_hyper_node} by setting $U=V_2$, since $S_U = \emptyset$ and $f_{U,g} = 0$. Furthermore, both $l_3$ and $l_5$ satisfy Theorem~\ref{lem:no_fa_hyper_direc} if $U=\{u_3,u_6,u_7\}$. As can be seen, some failed links ($l_2$ and $l_6$) can be covered by both Theorem~\ref{lem: final_lemma} and Theorem~\ref{lem:no_Miss_hyper_node}. Also, some operational links ($l_3$ and $l_5$) can be covered by both Theorem~\ref{lem: final_lemma} and Theorem~\ref{lem:no_fa_hyper_direc}. However, $l_1$ can only be covered by Theorem~\ref{lem: final_lemma}.
\end{example}

\subsubsection{Special Cases}

We can use these theorems to analyze the accuracy of Algorithm~\ref{alg: fedgeDet} in special cases of practical interest. \looseness=-1

\emph{No islanding:} We now examine the accuracy of Algorithm~\ref{alg: fedgeDet} in the special case that the grid remains connected after attack, which has been studied in \cite{Soltan18TCNS}. It is worthwhile to analyze 
this case (where $\bfDelta^*=\bm{0}$) because Algorithm~\ref{alg: fedgeDet} assumes no prior knowledge of $\bfDelta^*$ ({see proof in appendix}).

\begin{corollary}\label{cor:cond_connected}
If the grid stays connected after failure, $H$ is acyclic, and $H$ contains either no load bus or no generator bus, then Algorithm~\ref{alg: fedgeDet} can correctly detect $F$, i.e., $\hat{F} = F$.
\end{corollary}

{\emph{Remark:}} If the grid stays connected after failure, $H$ contains either no load bus or no generator bus, and $H$ contains cycles, then the status of a link $l\in E_H$ is guaranteed to be correctly identified if $l$ is not in any cycle. This is because in this case, we can construct a hyper-node satisfying the conditions in Theorem~\ref{lem:no_Miss_hyper_node} or \ref{lem:no_fa_hyper_direc} as in the proof of Corollary~\ref{cor:cond_connected}.

\emph{Islanding:}
Now, we study the case where the attacked area is decomposed into multiple islands. Suppose that the failures partition $H$ into $K$ islands $H_i = (V_i, E_i)$ ($i=1,\ldots,K$), where $V_H = \bigcup_{i=1}^K V_i$ 
and $E_H = (\bigcup_{i=1}^K E_i) \bigcup E_c$, with $E_c$ being the set of links between different islands. 
Then the following is implied by Theorem~\ref{lem: final_lemma} ({see proof in appendix}).

\begin{corollary}\label{col:no_gen_load}
Suppose that $F=E_c$. Let $E_{c,i} \subseteq E_c$ be the subset of failed links with one endpoint in $H_i$. Then, Algorithm~\ref{alg: fedgeDet} will correctly detect the failures ($\hat{F} = F$) if there exists a set $\mathcal{L} \subseteq \{H_i\}_{i=1}^K$ with $\bigcup_{H_i\in \mathcal{L}}E_{c,i} = E_c$, such that each $H_i\in \mathcal{L}$ satisfies the following conditions\footnote{See \eqref{eq:properties of hyper-node} {and Tabel~\ref{tab:notation}} for the definitions of notations. 
}:
\begin{enumerate}
    \item $\forall e, e'\in E_{c,i}$, $\tilde{D}_{V_i,e}\tilde{D}_{V_i,e'} > 0$,
    \item $\forall v\in V_i$ that is incident to a link in $l\in E_{c,i}$, $f_{\{v\},0} < f_{\{v\},1}$, and
    \item $\forall v\in V_i$ that is incident to a link in $l\in E_{c,i}$, $g_{D,v} = 0$ if $\tilde{D}_{v,l} < 0$, and $g_{D,-v} = 0$ if $\tilde{D}_{v,l} > 0$.
\end{enumerate}
\end{corollary}

\emph{Remark:} Corollary~\ref{col:no_gen_load} extends Theorem~V.1 in~\cite{yudi20SmartGridComm} in the sense that Corollary~\ref{col:no_gen_load} only requires hypothetical power flows on failed links to be in the same direction, while Theorem~V.1 of \cite{yudi20SmartGridComm} requires these power flows to have a specific direction.

\section{Performance Evaluation}\label{sec:Performance Evaluation}

We test our solutions on the Polish power grid (``Polish system - winter 1999-2000 peak'') \cite{zimmerman2019matpower} 
with $2383$ nodes and $2886$ links, 
where parallel links are combined into one link. We generate the attacked area $H$ by randomly choosing one node as a starting point and performing a breadth first search to obtain $H$ with a predetermined $|V_H|$. We then randomly choose $|F|$ links within $H$ to fail.
We vary $|V_H|$ and $|F|$ to explore different settings, and for each setting, we generate $70$ different $H$'s and $300$ different $F$'s per $H$. 
{In our experiments, $|V_H|$ is chosen to be large enough such that there are sufficiently many internal nodes in $H$ whose post-attack active power injections cannot be easily recovered, and there are sufficiently many candidate links to fail that can lead to island formation in the grid.}

We evaluate two types of metrics: (1) how often the recovery conditions are satisfied, and (2) how accurate our algorithm is when its recovery conditions are not necessarily satisfied. \looseness=-1

\begin{figure}
    \centerline{
    \includegraphics[width=.48\linewidth]{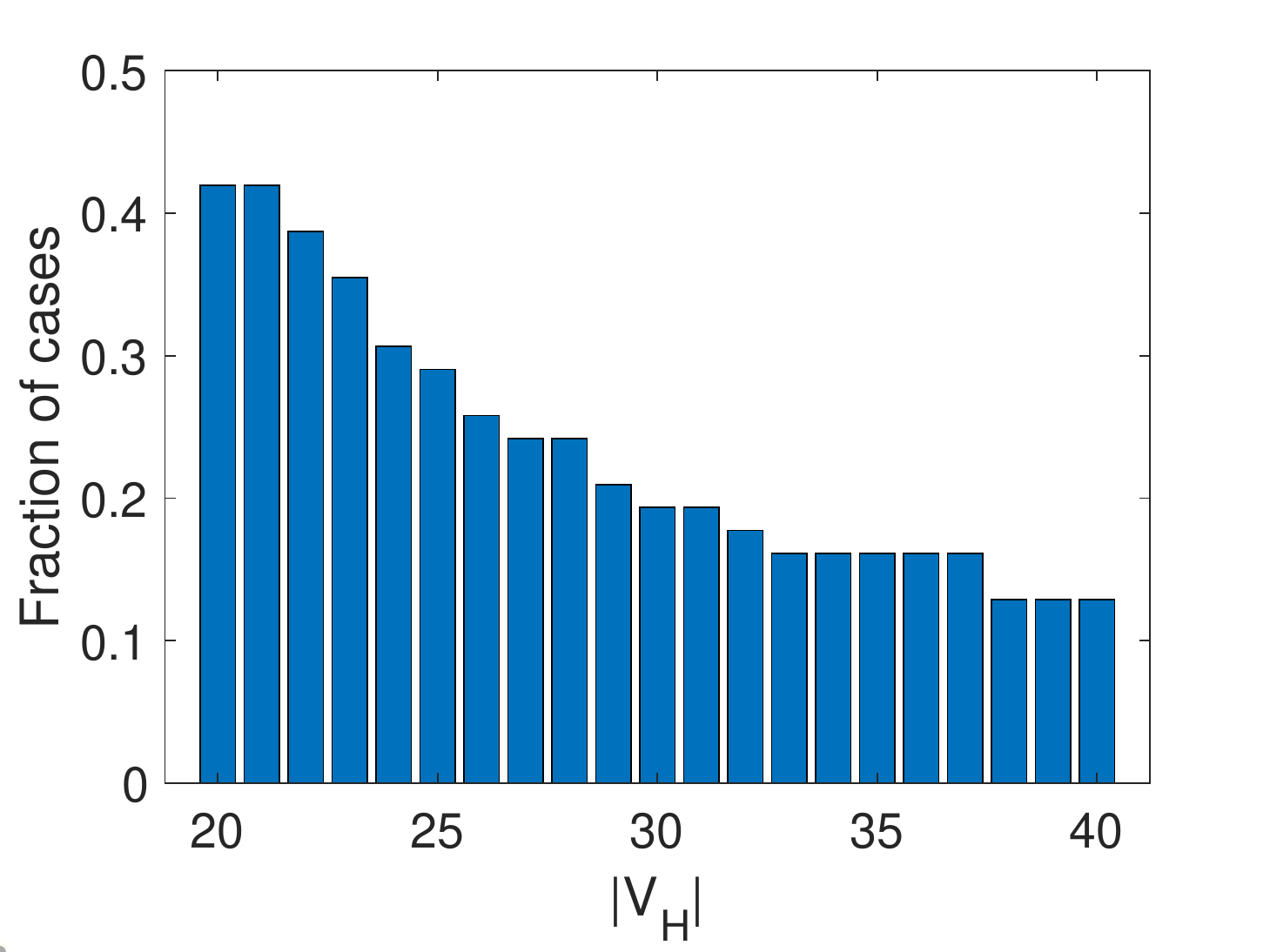}}
    \vspace{-1em}
    \caption{Prob. that condition of Theorem~\ref{thm:localize failed links}.(1) holds. }
    \label{fig:T41_acyclic}
\vspace{-1em}
\end{figure}

\subsection{Probability of Guaranteed Recovery}

First, we evaluate the fraction of randomly generated cases (combinations of $H$ and $F$) satisfying the conditions in Theorem~\ref{thm:recovery of phase angles} for recovering the phase angles, and Theorem~\ref{thm:localize failed links}.(1)\footnote{We only tested condition (1) in Theorem~\ref{thm:localize failed links}, as the other condition relies on complicated graph properties that are difficult to test. } for localizing the failed links with known phase angles and active powers. We observe that: (i) the condition in Theorem~\ref{thm:recovery of phase angles} is almost never satisfied under a nontrivial size of $H$ ($|V_H|=20,\ldots,40$), which emphasizes the importance of securing PMU measurements; (ii) the condition in Theorem~\ref{thm:localize failed links}.(1) is only satisfied with a small probability as shown Fig.~\ref{fig:T41_acyclic}, which decreases with the size of $H$ (note that Theorem~\ref{thm:localize failed links}.(1) does not depend on $F$).
In previous work~\cite{Soltan18TCNS}, this issue was addressed by actively designing ``zones'' for reporting measurements such that each zone satisfies these stringent conditions, which only works for attacks limited to single zones. 
In contrast, our recovery conditions are much more general as shown below (Fig.~\ref{fig:comp_theo_exp_h40_bar3_miss} and Fig.~\ref{fig:comp_theo_exp_h40_bar3_fa}), even though we do not assume to know the post-attack power injections in $H$. \looseness=-1

\begin{figure}
\begin{minipage}{.495\linewidth}\label{subfig:missLinks_LP_H40}
  \centerline{
  \includegraphics[width=1\columnwidth]{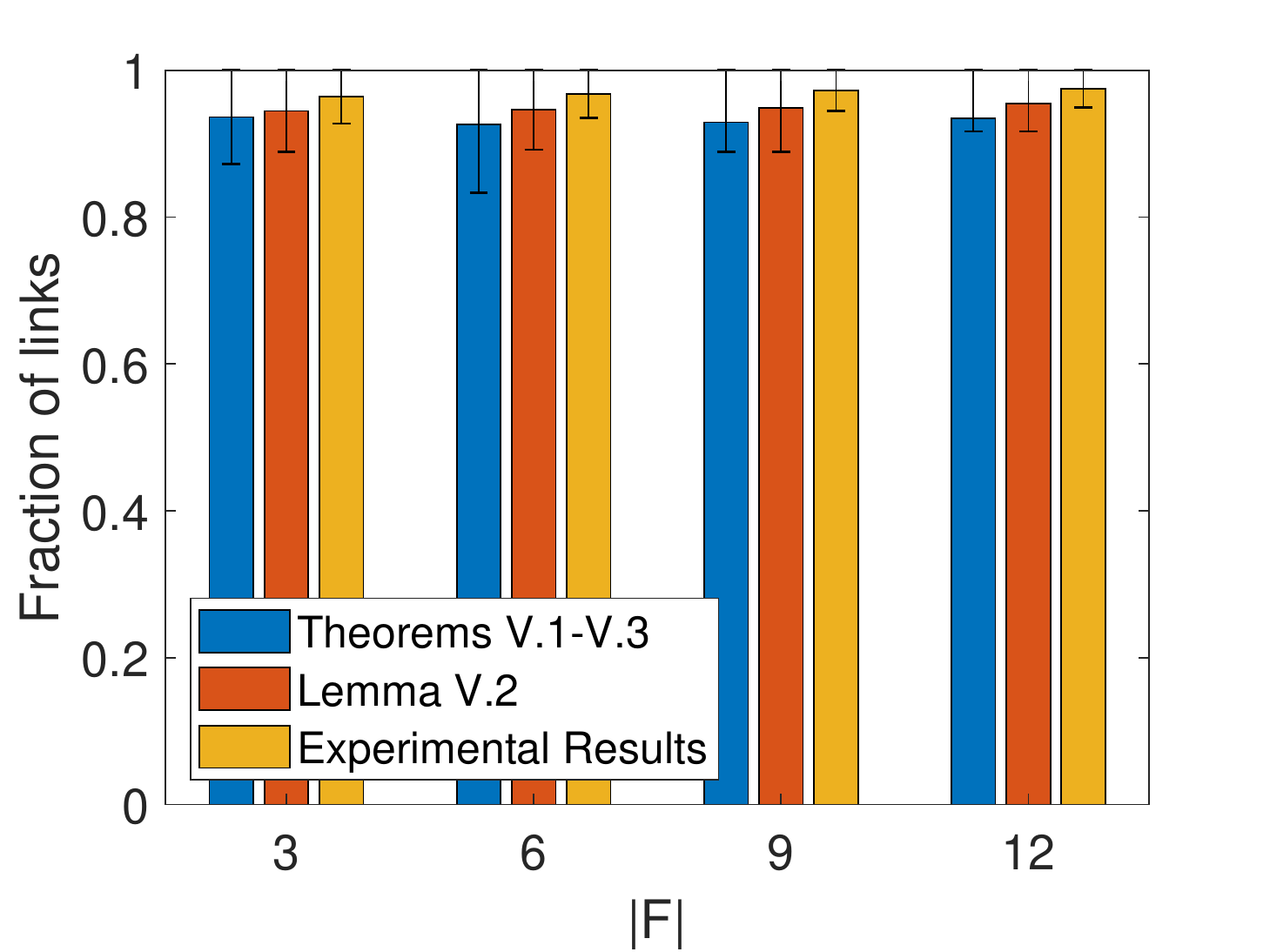}
  }
  \centerline{\small (a) Fraction of failed links.}
\end{minipage}\hfill
\begin{minipage}{.495\linewidth}\label{subfig:Journal_PerCases_miss_H40}
 \centerline{
  \includegraphics[width=1\columnwidth]{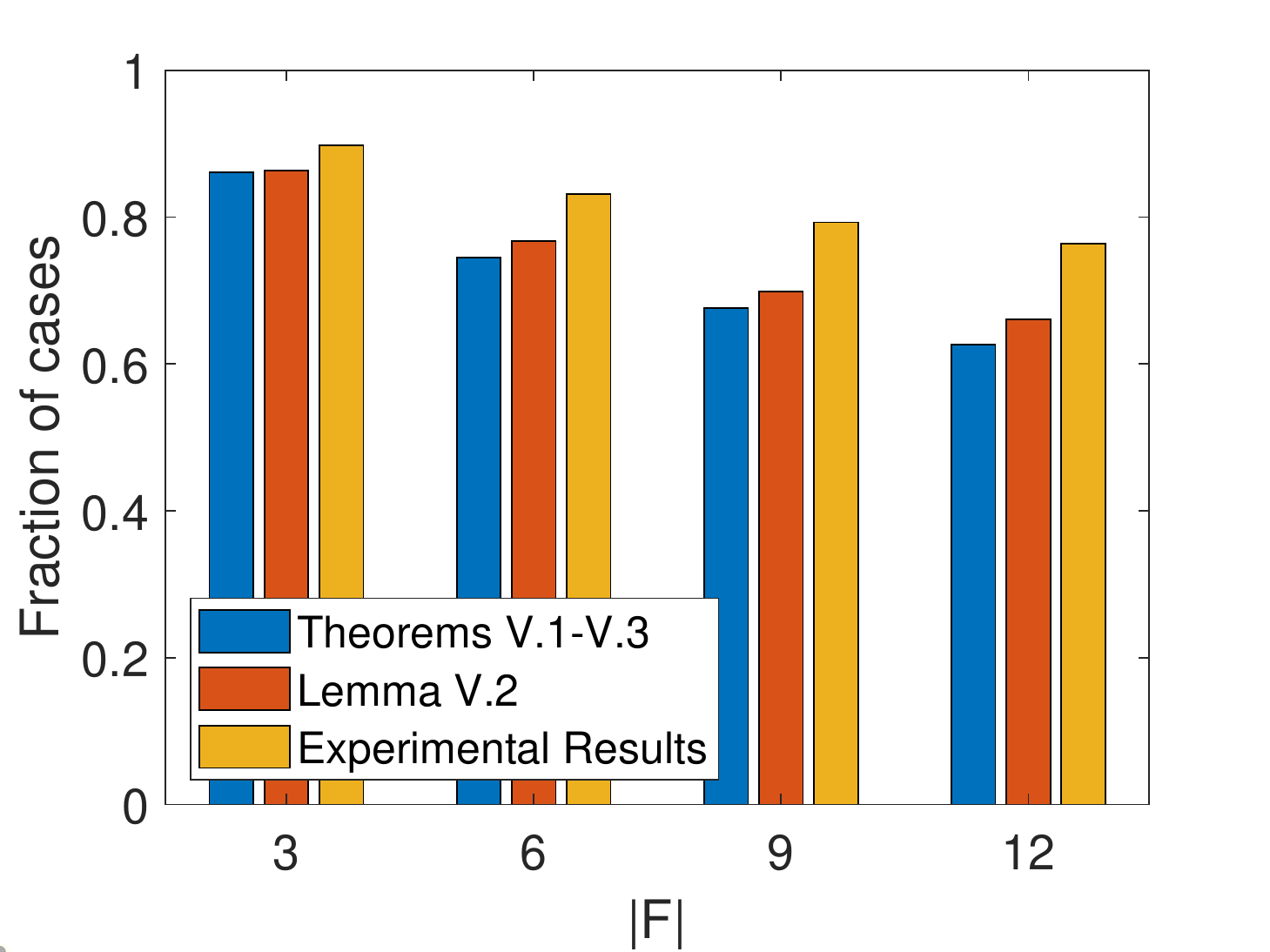}
  }
  \centerline{\small (b) Prob. of no miss}
\end{minipage}
  \caption{Prob. of correctly identifying failed link status with known phase angles but unknown active powers ($|V_H|=40$). }
  \label{fig:comp_theo_exp_h40_bar3_miss}
    \vspace{-1em}
\end{figure}

\begin{figure}
\vspace{-1.5em}
\begin{minipage}{.495\linewidth}\label{subfig:comp_miss}
  \centerline{
  \includegraphics[width=1\columnwidth]{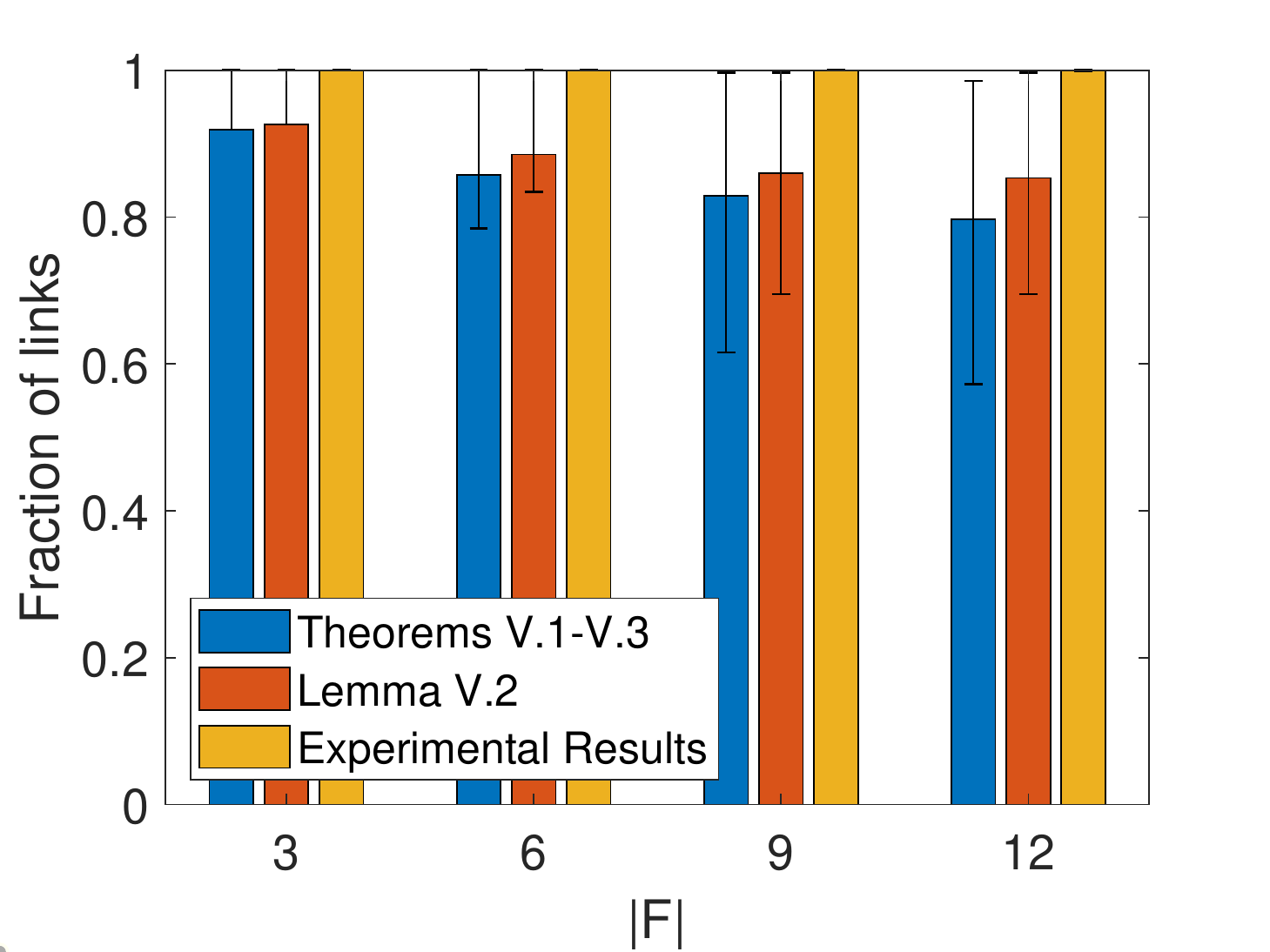}
  }
  \centerline{\small (a) Fraction of operational links.}
\end{minipage}\hfill
\begin{minipage}{.495\linewidth}\label{subfig:comp_fa}
 \centerline{
  \includegraphics[width=1\columnwidth]{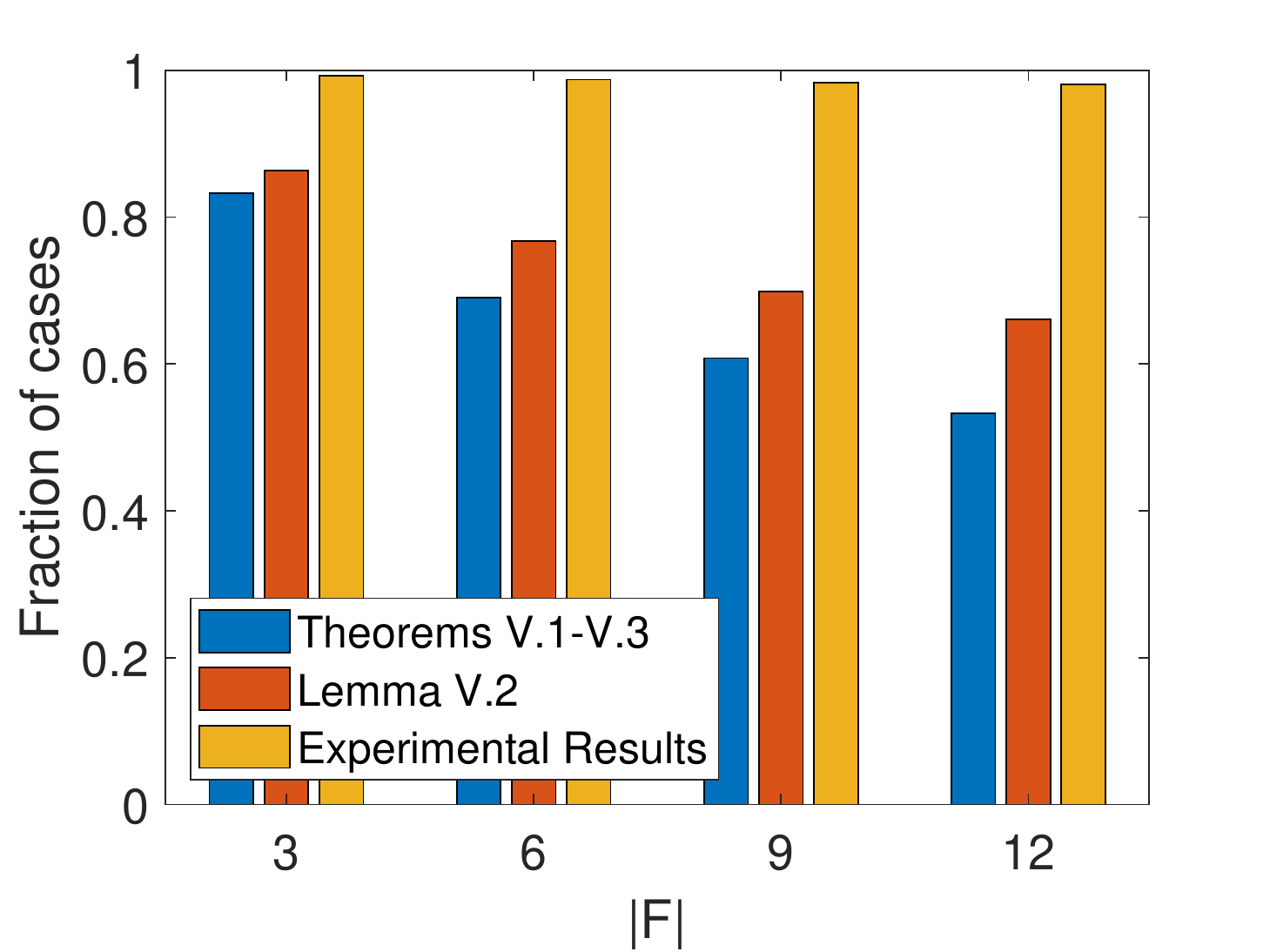}
  }
  \centerline{\small (b) Prob. of no false alarm.}
\end{minipage}
  \caption{Prob. of correctly identifying operational link status with known phase angles but unknown active powers ($|V_H|=40$). }
  \label{fig:comp_theo_exp_h40_bar3_fa}
\vspace{-1em}
\end{figure}

Then, we evaluate the fraction of links satisfying the recovery conditions in Lemma~\ref{lem:ground_alter_gale} and Theorems~\ref{lem:no_Miss_hyper_node}--\ref{lem: final_lemma}, together with the actual fraction of links whose status are correctly identified by Algorithm~\ref{alg: fedgeDet}, under known phase angles and unknown active powers. The results are shown in Fig.~\ref{fig:comp_theo_exp_h40_bar3_miss} and Fig.~\ref{fig:comp_theo_exp_h40_bar3_fa}, where ``Experimental Results'' is the actual performance of Algorithm~\ref{alg: fedgeDet}, “Lemma V.2” indicates the failed/operational links satisfying Lemma~\ref{lem:ground_alter_gale}, and ``Theorems~V.1--V.3'' indicates the failed links that satisfy either Theorem~\ref{lem:no_Miss_hyper_node} or \eqref{eq: no_miss_hyper_f} in Theorem~\ref{lem: final_lemma}, or the operational links that satisfy either Theorem~\ref{lem:no_fa_hyper_direc} or \eqref{eq: no_fa_hyper_f} in Theorem~\ref{lem: final_lemma}. More specifically, Figs.~\ref{fig:comp_theo_exp_h40_bar3_miss}--\ref{fig:comp_theo_exp_h40_bar3_fa}~(a) {show} the fraction of correctly identified failed/operational links, averaged over all the randomly generated cases, where the bottom and the top edges of the error bar indicate the $25^{\mbox{th}}$ and $75^{\mbox{th}}$ percentiles. In addition, Figs.~\ref{fig:comp_theo_exp_h40_bar3_miss}--\ref{fig:comp_theo_exp_h40_bar3_fa}~(b) {demonstrate} the fraction of cases with no miss/false alarm.

It is worth noting that checking whether Theorems~\ref{lem:no_Miss_hyper_node}--\ref{lem: final_lemma} hold for each link is hard since {they require} testing all possible hyper-nodes, whose number is exponential in $|V_H|$. To test whether Theorem~\ref{lem:no_Miss_hyper_node} holds for a given link, we heuristically construct a hyper-node through BFS starting from each of its endpoints. In each iteration of BFS, we add all the nodes that cause violation of condition~1 or condition~2 in Theorem~\ref{lem:no_Miss_hyper_node} into the hyper-node and test whether condition~3 holds. This procedure applies similarly to the testing of Theorem~\ref{lem:no_fa_hyper_direc} and the construction of $T$ in Theorem~\ref{lem: final_lemma}.
Therefore, the fractions of links/cases covered by {Theorems~\ref{lem:no_Miss_hyper_node}--\ref{lem: final_lemma}} in Figs.~\ref{fig:comp_theo_exp_h40_bar3_miss}--\ref{fig:comp_theo_exp_h40_bar3_fa} are {lower bounds}.
{Meanwhile, checking whether Lemma~\ref{lem:ground_alter_gale} holds for each link is easy. }
Nevertheless, we see that (\romannumeral1) Theorems~\ref{lem:no_Miss_hyper_node}-\ref{lem: final_lemma} explain the success of Algorithm~\ref{alg: fedgeDet} quite well, especially when $|F|$ is small, (\romannumeral2) {the theorems can explain most cases covered by Lemma~\ref{lem:ground_alter_gale}}, and (\romannumeral3) Algorithm~\ref{alg: fedgeDet} is actually highly accurate in identifying the operational links even though the theoretical conditions for doing so appear stringent.
{To better understand the last phenomenon, we observe in experiments that many operational links carry small post-attack power flows, which makes the conditions in Theorem~\ref{lem:no_fa_hyper_direc} and \eqref{eq: no_fa_hyper_f} in Theorem~\ref{lem: final_lemma} hard to satisfy. On the contrary, the values of hypothetical power flows on failed links are usually large, making the conditions in  Theorem~\ref{lem:no_Miss_hyper_node} and \eqref{eq: no_miss_hyper_f} in Theorem~\ref{lem: final_lemma} easier to satisfy. }
\looseness=-1 

\begin{figure}
\begin{minipage}{.48\linewidth}
    \centerline{
    \includegraphics[width=1\columnwidth]{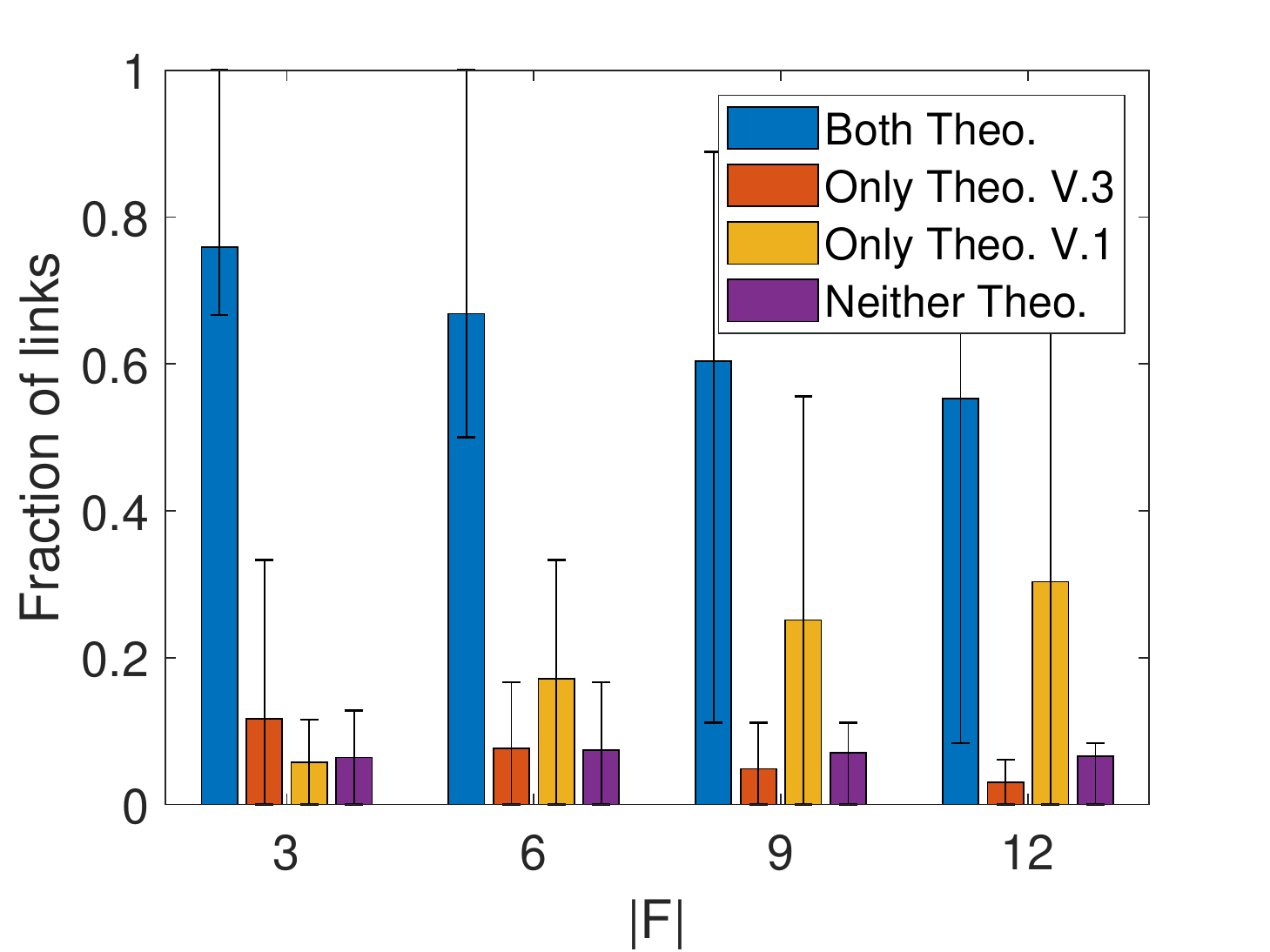}}
    \vspace{-.5em}
    \centerline{\small (a) failed links }
\end{minipage}\hfill
\begin{minipage}{.48\linewidth}
    \centerline{
    \includegraphics[width=1\columnwidth]{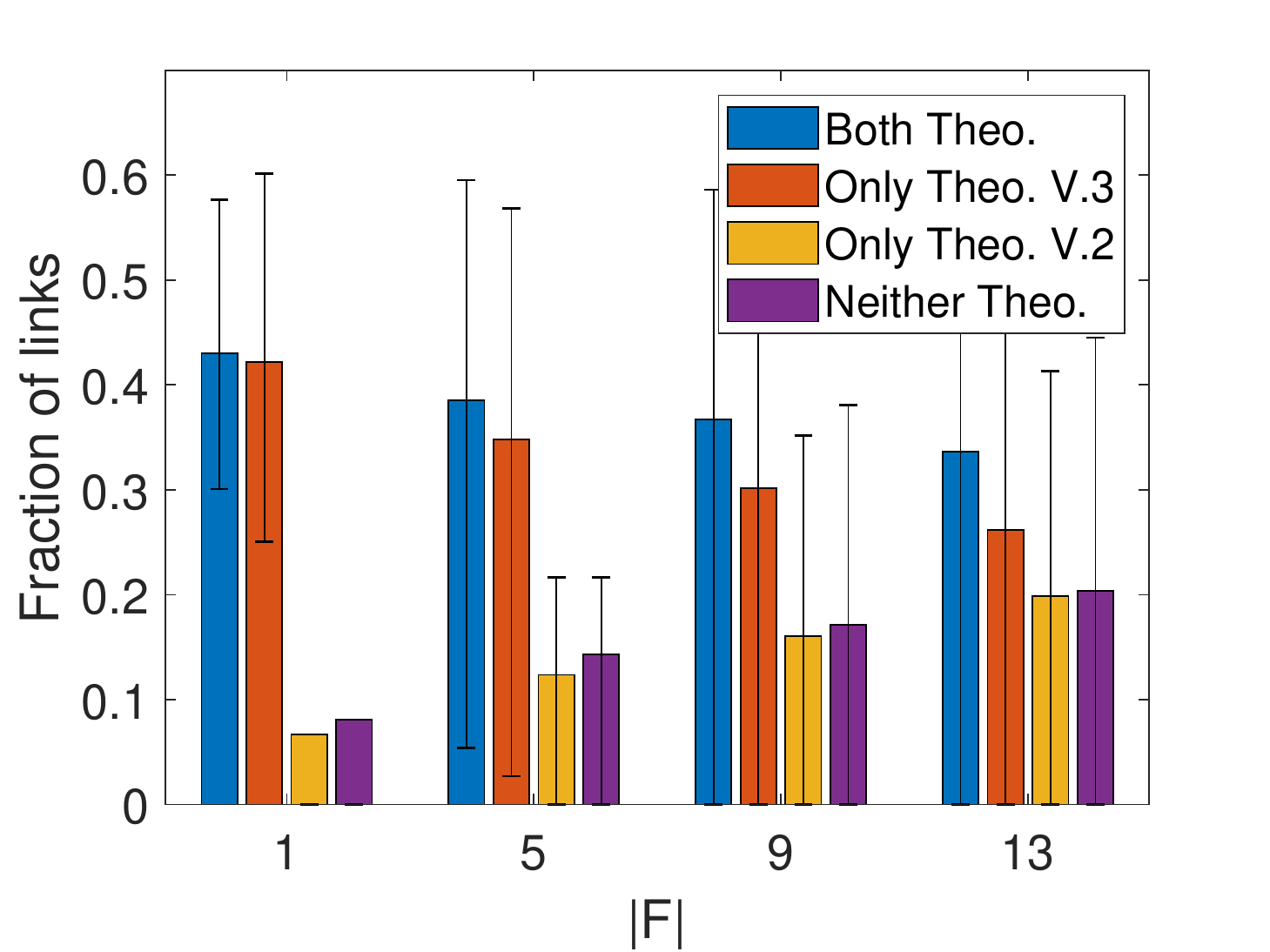}}
        \vspace{-.5em}
    \centerline{\small (b) operational links}
\end{minipage}
\vspace{-.05em}
\caption{Prob. that conditions of Theorems~\ref{lem:no_Miss_hyper_node}--\ref{lem: final_lemma} hold.  }\label{fig:relationship}
\vspace{-1em}
\end{figure}

To better understand the relationship between Theorems~\ref{lem:no_Miss_hyper_node}--\ref{lem: final_lemma}, 
we decompose all the failed links into 4 categories: (1) links satisfying both Theorem~\ref{lem:no_Miss_hyper_node} and \eqref{eq: no_miss_hyper_f} in Theorem~\ref{lem: final_lemma}, (2) links satisfying only \eqref{eq: no_miss_hyper_f} in Theorem~\ref{lem: final_lemma}, (3) links satisfying only Theorem~\ref{lem:no_Miss_hyper_node}, and (4) links satisfying neither. Fig.~\ref{fig:relationship}~(a) shows the fraction of links in each category, averaged over all the simulated cases. We observe that (\romannumeral1) many failed links satisfy both conditions; (\romannumeral2) as $|F|$ grows, the fraction of failed links satisfying only \eqref{eq: no_miss_hyper_f} in Theorem~\ref{lem: final_lemma} decreases, while the fraction of links satisfying only Theorem~\ref{lem:no_Miss_hyper_node} increases; (\romannumeral3) the fraction of links covered by neither of the conditions is small. Similarly, we look into the coverage of Theorem~\ref{lem:no_fa_hyper_direc} and \eqref{eq: no_fa_hyper_f} in Theorem~\ref{lem: final_lemma} for operational links, as shown in Fig.~\ref{fig:relationship}~(b), with similar observations 
except that the fraction of links covered by neither of these conditions increases with $|F|$. \looseness=-1 

\subsection{Accuracy of Failure Localization}

\begin{figure}
    \centerline{
    \includegraphics[width=.48\linewidth]{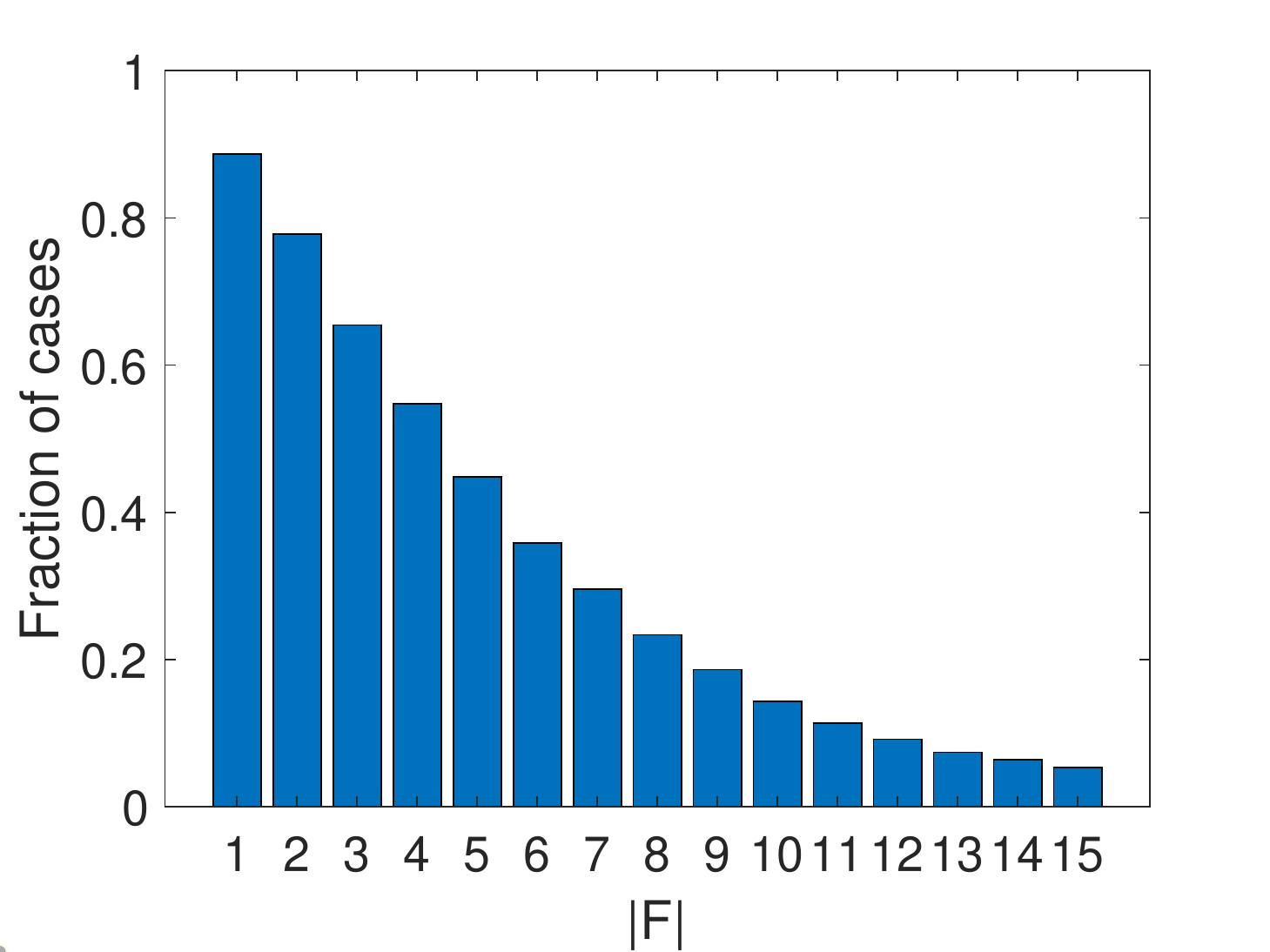}}
        \vspace{-1em}
    \caption{Prob. that setting $\bfDelta = \bm{0}$ is feasible ($|V_H|=40$).}
    \label{fig:T41}
\vspace{-1em}
\end{figure}

Next, we compare Algorithm~\ref{alg: fedgeDet} with benchmarks in localizing the failed links, assuming that post-attack phase angles are known, due to the better protection of PMU measurements~\cite{WAMPACsecurity}.
We consider two benchmarks: (\romannumeral1) the solution given in Theorem~\ref{thm:localize failed links} (extended from \cite{Soltan18TCNS}), i.e., estimating $F$ by $\supp(\bfx)$ for the solution to $\min \|\bfx\|_1$ s.t. (\ref{eq:D_H x equation}), assuming the true $\bfDelta_H$ to be known, and (\romannumeral2) $\min \|\bfx\|_1$ s.t. $\Arrowvert \bfB_{H|G}(\bftheta - \bftheta') - \bfD_H \bfx \Arrowvert_2  \le \Arrowvert \bfP_H \Arrowvert_2$, which is extended from the solutions in \cite{Zhu12TPS, chen2014efficient}.
We note that the original solution in \cite{Soltan18TCNS} (which assumes $\bfDelta=\bm{0}$) is often infeasible for our problem, as shown in Fig.~\ref{fig:T41}, thus not used as a benchmark.
Note that benchmark~(\romannumeral1) should be treated as a \emph{``performance upper bound''}, as it assumes more knowledge (i.e., ground-truth $\bfDelta_H$)
than our proposed algorithm. \looseness=-1 

As shown in Fig.~\ref{fig:comp_miss_h40_bar3}, benchmark~(\romannumeral1) demonstrates the best performance with regard to both miss-detection rate and the probability of having no miss-detection, while Algorithm~\ref{alg: fedgeDet} performs much better than benchmark~(\romannumeral2). This confirms the importance of knowing or estimating load shedding values in failure localization. Regarding the false alarm as shown in Fig.~\ref{fig:comp_fa_h40_bar3}, Algorithm~\ref{alg: fedgeDet} performs even better than benchmark~(\romannumeral1). This is because the decision variable $\bfx$ in benchmark~(\romannumeral1) combines the information of both the failed links and the phase angles $\bftheta_H'$, and thus does not fully exploit the knowledge of $\bftheta_H'$.
Furthermore, from the specific number of false alarms/misses in Fig.~\ref{fig:fa_miss_h40_bar3}, we see that Algorithm~\ref{alg: fedgeDet} correctly detects all the failed links with almost no false alarm for the majority of the time, while only missing a couple of failed links for the rest.  \looseness=-1



\begin{figure}
\begin{minipage}{.495\linewidth}
  \centerline{
  \includegraphics[width=1\columnwidth]{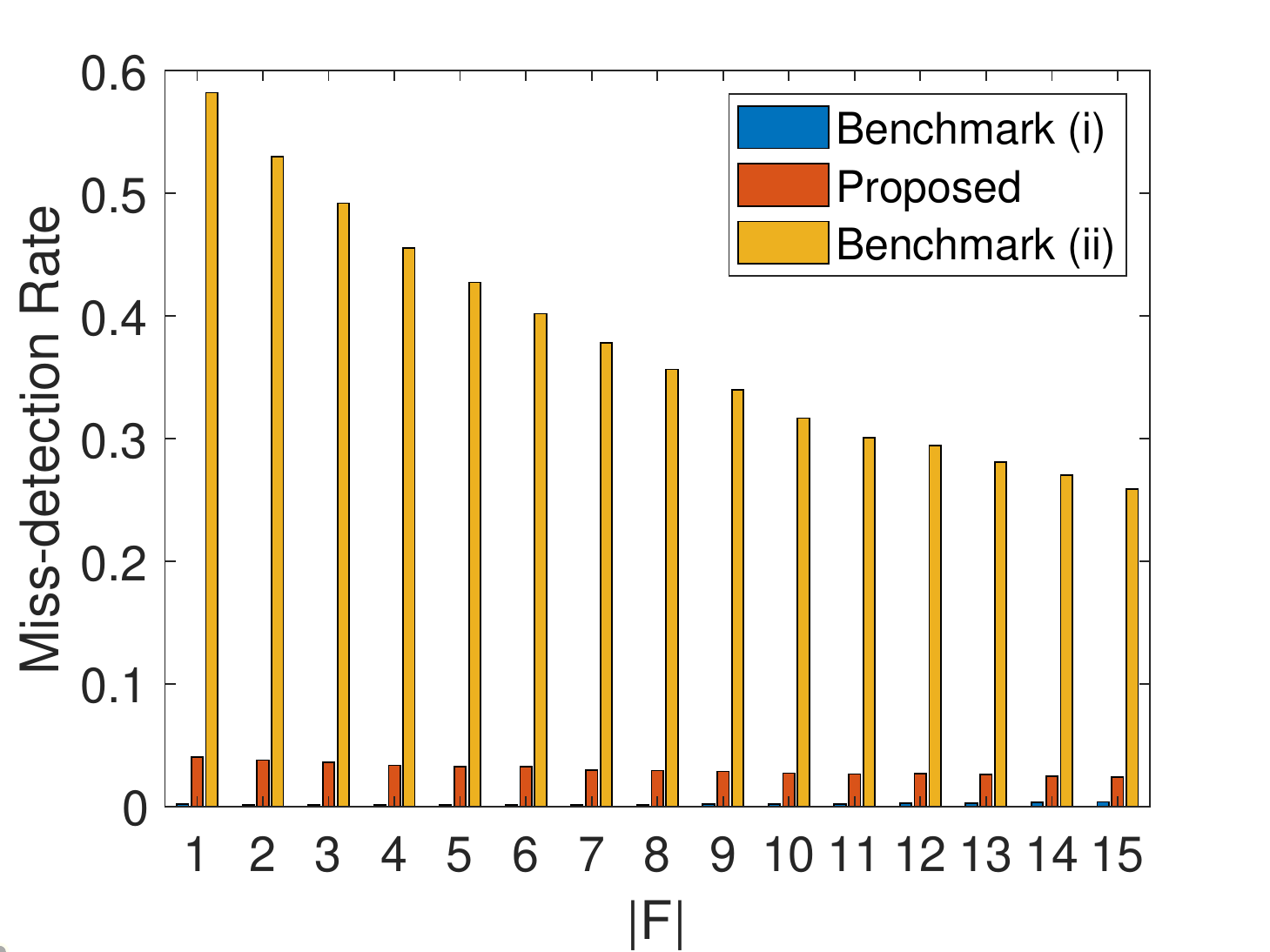}}
  \centerline{\small (a) miss rate}
\end{minipage}\hfill
\begin{minipage}{.495\linewidth}
 \centerline{
  \includegraphics[width=1\columnwidth]{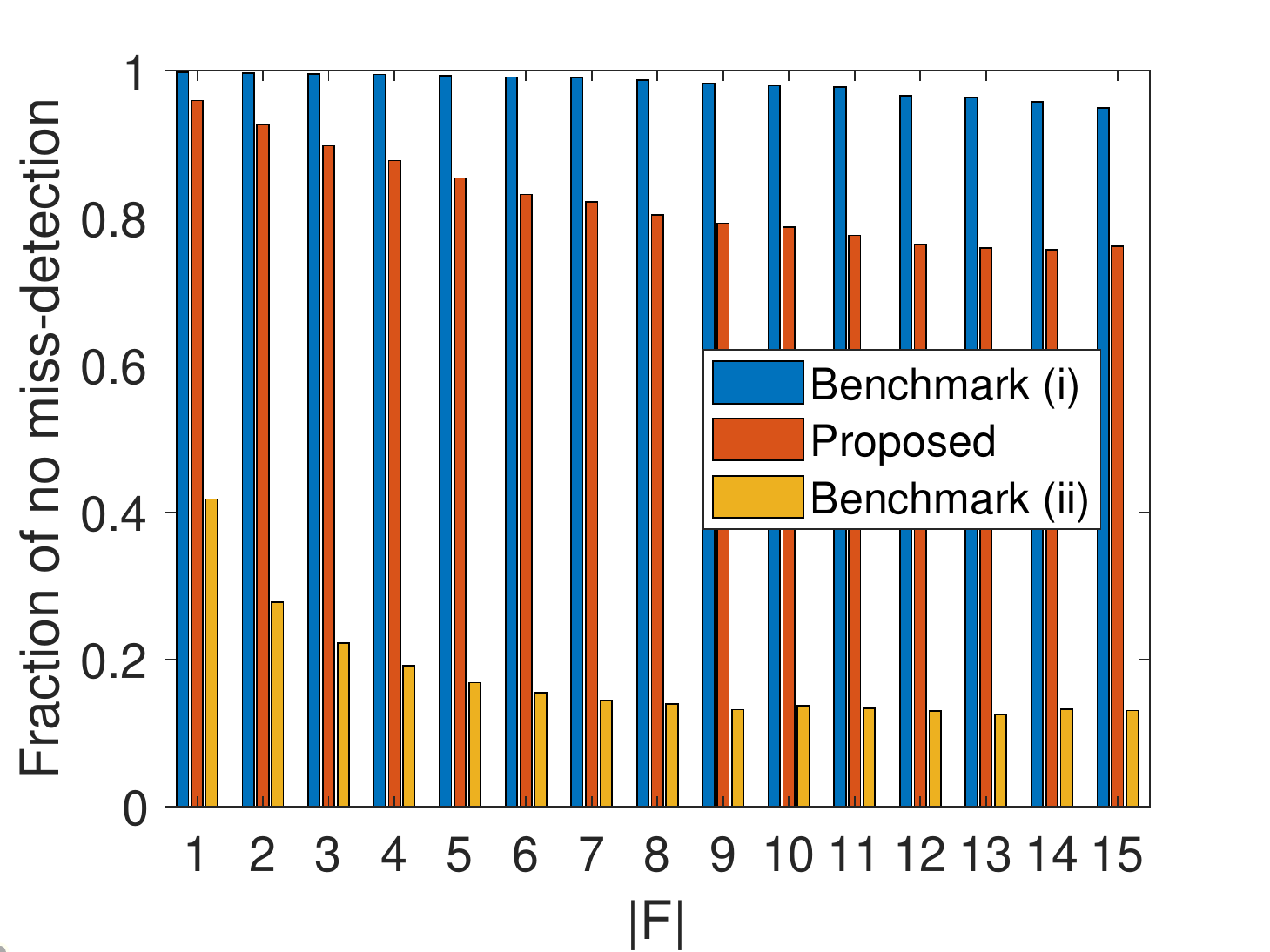}}
  \centerline{\small (b) Probability of no miss}
\end{minipage}
  \caption{Performance comparison on miss rate ($|V_H|=40$).}
  \label{fig:comp_miss_h40_bar3}
  \vspace{-1em}
\end{figure}

\begin{figure}
\begin{minipage}{.495\linewidth}
  \centerline{
  \includegraphics[width=1\columnwidth]{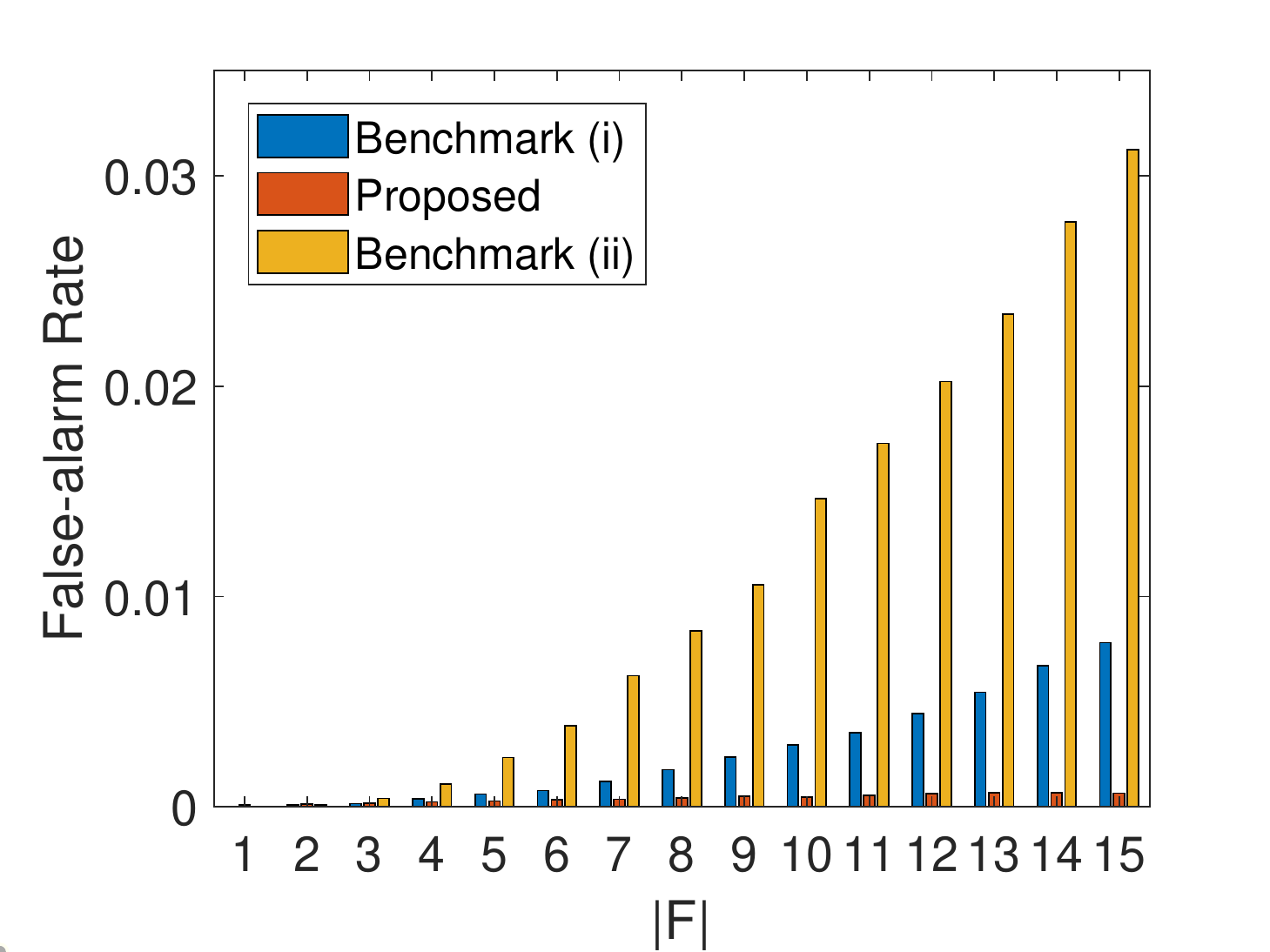}}
  \centerline{\small (a) false alarm rate}
\end{minipage}\hfill
\begin{minipage}{.495\linewidth}
 \centerline{
  \includegraphics[width=1\columnwidth]{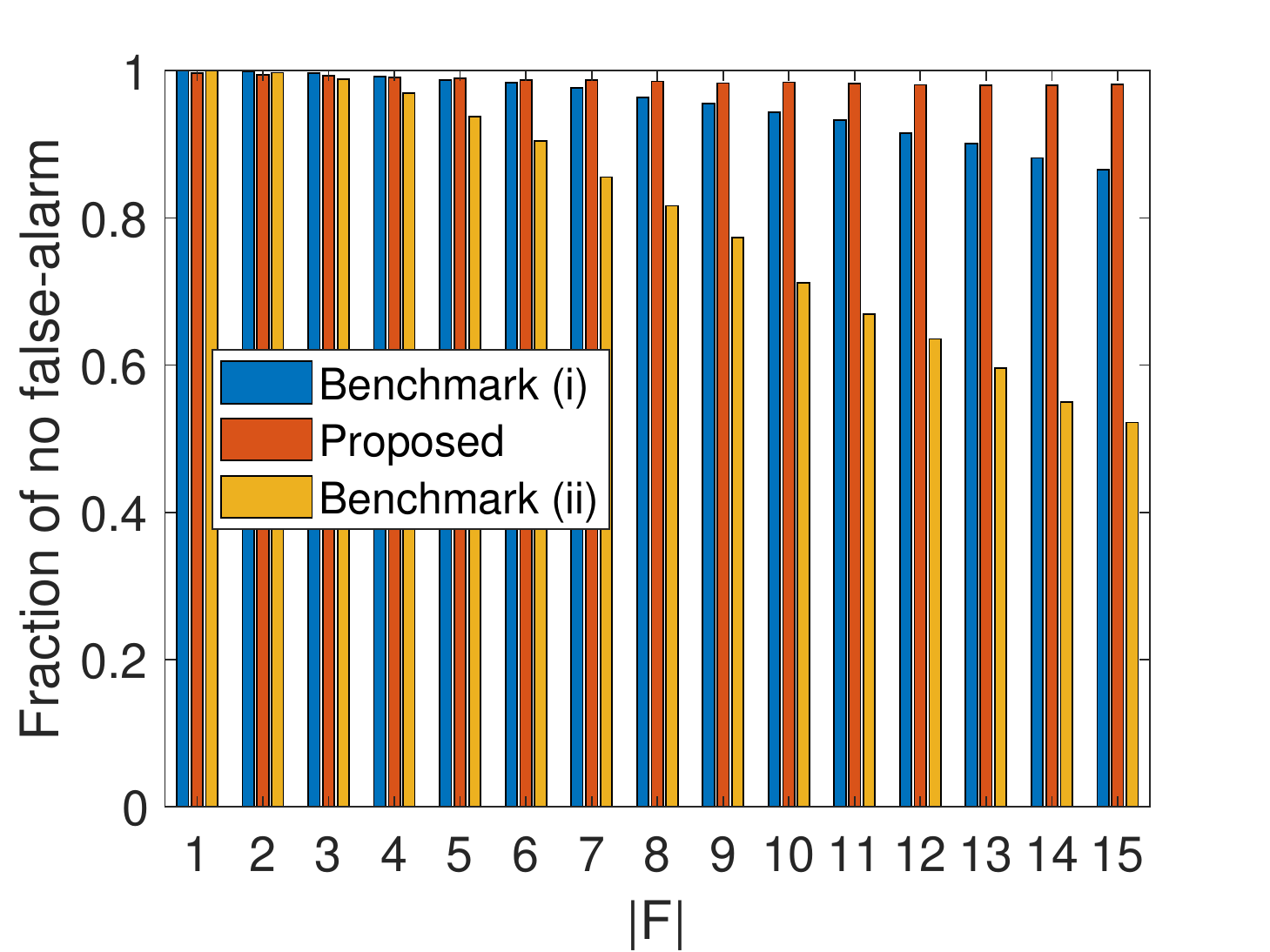}}
  \centerline{\small (b) Probability of no false alarm}
\end{minipage}
  \caption{Performance comparison on false alarm rate ($|V_H|=40$).}
  \label{fig:comp_fa_h40_bar3}
  \vspace{-1em}
\end{figure}

\begin{figure}
\begin{minipage}{.495\linewidth}
  \centerline{
  \includegraphics[width=1\columnwidth]{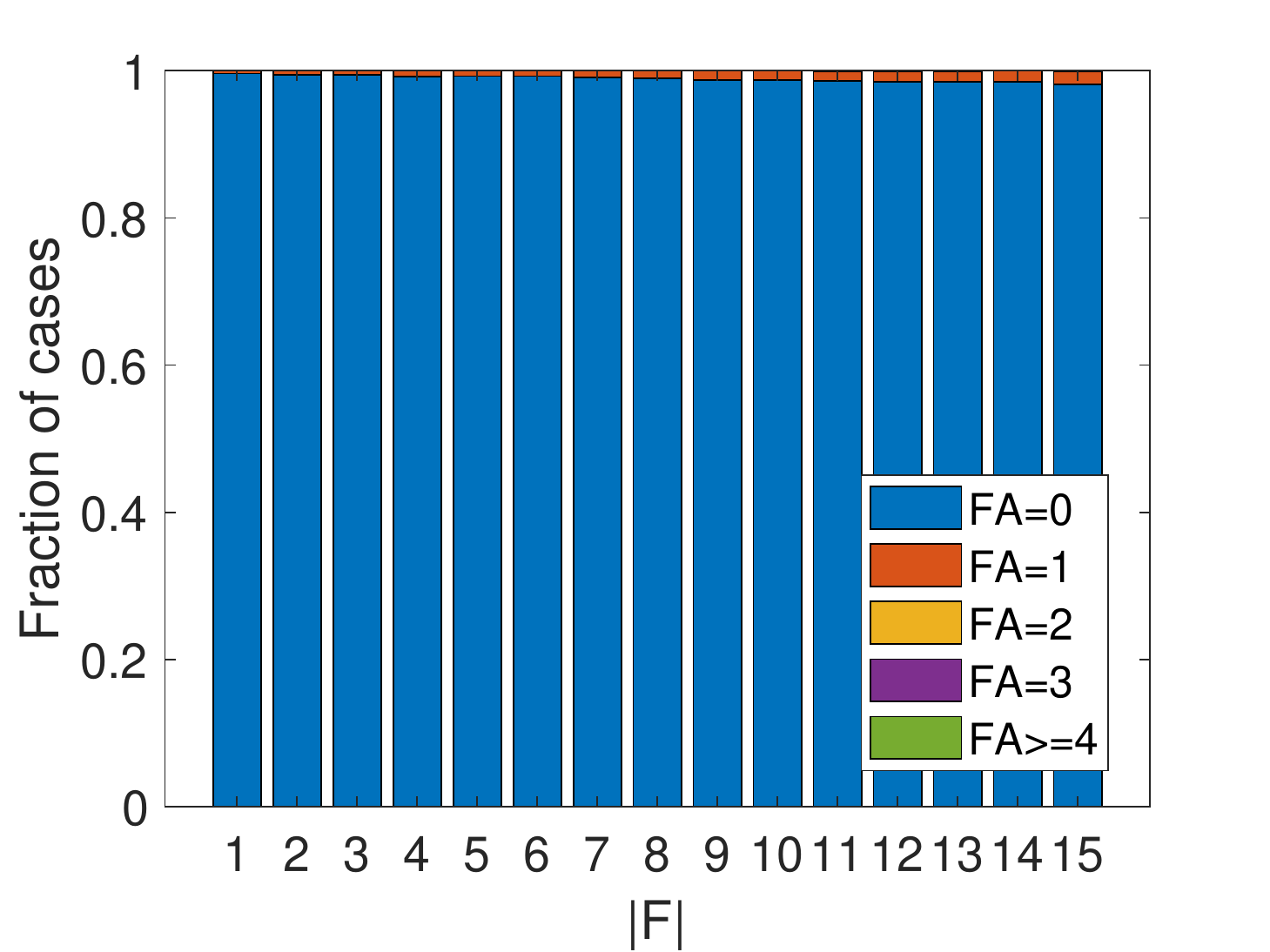}}
  \centerline{\small (a) false alarm }
\end{minipage}\hfill
\begin{minipage}{.495\linewidth}
 \centerline{
  \includegraphics[width=1\columnwidth]{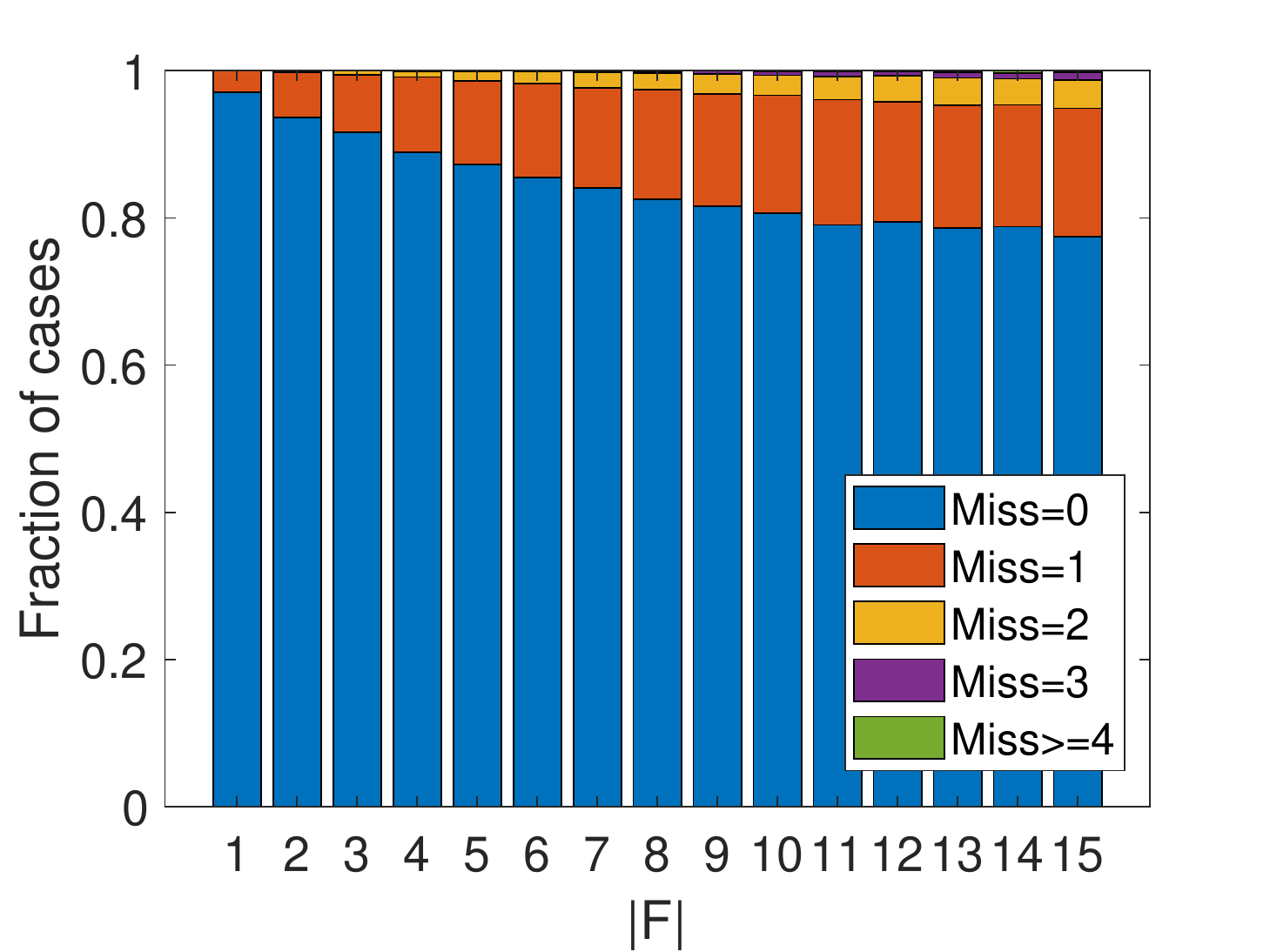}}
  \centerline{\small (b) miss }
\end{minipage}
  \caption{Number of false alarms/misses of Algorithm~\ref{alg: fedgeDet} ($|V_H|=40$).}
  \label{fig:fa_miss_h40_bar3}
  \vspace{-1em}
\end{figure}



\section{Conclusion}\label{sec:Conclusion}

We investigated the problem of power grid state estimation under general cyber-physical attack that may disconnect the grid. First, we demonstrated that existing solutions and the corresponding recovery conditions for recovering phase angles and link status are still applicable with the knowledge of post-attack power injections. Second, for unknown post-attack power injections, we proposed an LP-based algorithm to identify link status within the attacked area with the knowledge of recovered phase angles. We established sufficient conditions for the proposed algorithm to identify the link status correctly. Finally, our evaluations based on the Polish power grid demonstrated that the proposed algorithm is highly accurate in localizing the failed links, and the majority of cases are theoretically guaranteed by the presented conditions. Meanwhile, the existing condition for perfectly recovering phase angles can be hard to satisfy in practice, indicating the importance of safeguarding PMU measurements from cyber attacks.  \looseness=-1


\bibliographystyle{IEEEtran}
\bibliography{myBib}

\begin{thebibliography}{10}
\providecommand{\url}[1]{#1}
\csname url@samestyle\endcsname
\providecommand{\newblock}{\relax}
\providecommand{\bibinfo}[2]{#2}
\providecommand{\BIBentrySTDinterwordspacing}{\spaceskip=0pt\relax}
\providecommand{\BIBentryALTinterwordstretchfactor}{4}
\providecommand{\BIBentryALTinterwordspacing}{\spaceskip=\fontdimen2\font plus
\BIBentryALTinterwordstretchfactor\fontdimen3\font minus
  \fontdimen4\font\relax}
\providecommand{\BIBforeignlanguage}[2]{{%
\expandafter\ifx\csname l@#1\endcsname\relax
\typeout{** WARNING: IEEEtran.bst: No hyphenation pattern has been}%
\typeout{** loaded for the language `#1'. Using the pattern for}%
\typeout{** the default language instead.}%
\else
\language=\csname l@#1\endcsname
\fi
#2}}
\providecommand{\BIBdecl}{\relax}
\BIBdecl

\bibitem{yudi20SmartGridComm}
Y.~Huang, T.~He, N.~R. Chaudhuri, and T.~L. Porta, ``Power grid state
  estimation under general cyber-physical attacks,'' in \emph{IEEE
  SmartGridComm}.\hskip 1em plus 0.5em minus 0.4em\relax IEEE, 2020.

\bibitem{Fairley16Spectrum}
P.~Fairley, ``Cybersecurity at {U.S.} utilities due for an upgrade: Tech to
  detect intrusions into industrial control systems will be mandatory,''
  \emph{IEEE Spectrum}, vol.~53, no.~5, pp. 11--13, May 2016.

\bibitem{Soltan18TCNS}
S.~Soltan, M.~Yannakakis, and G.~Zussman, ``Power grid state estimation
  following a joint cyber and physical attack,'' \emph{IEEE Transactions on
  Control of Network Systems}, vol.~5, no.~1, pp. 499--512, 2018.

\bibitem{Soltan17PES}
S.~Soltan and G.~Zussman, ``Power grid state estimation after a cyber-physical
  attack under the {AC} power flow model,'' in \emph{IEEE PES-GM}, 2017.

\bibitem{huang2012state}
Y.-F. Huang, S.~Werner, J.~Huang, N.~Kashyap, and V.~Gupta, ``State estimation
  in electric power grids: Meeting new challenges presented by the requirements
  of the future grid,'' \emph{IEEE Signal Processing Magazine}, vol.~29, no.~5,
  pp. 33--43, 2012.

\bibitem{liu2011false}
Y.~Liu, P.~Ning, and M.~K. Reiter, ``False data injection attacks against state
  estimation in electric power grids,'' \emph{ACM Transactions on Information
  and System Security (TISSEC)}, vol.~14, no.~1, pp. 1--33, 2011.

\bibitem{shoukry2017secure}
Y.~Shoukry, P.~Nuzzo, A.~Puggelli, A.~L. Sangiovanni-Vincentelli, S.~A. Seshia,
  and P.~Tabuada, ``Secure state estimation for cyber-physical systems under
  sensor attacks: A satisfiability modulo theory approach,'' \emph{IEEE
  Transactions on Automatic Control}, vol.~62, no.~10, pp. 4917--4932, 2017.

\bibitem{dan2010stealth}
G.~Dan and H.~Sandberg, ``Stealth attacks and protection schemes for state
  estimators in power systems,'' in \emph{2010 first IEEE international
  conference on smart grid communications}.\hskip 1em plus 0.5em minus
  0.4em\relax IEEE, 2010, pp. 214--219.

\bibitem{vukovic2011network}
O.~Vukovi{\'c}, K.~C. Sou, G.~D{\'a}n, and H.~Sandberg, ``Network-layer
  protection schemes against stealth attacks on state estimators in power
  systems,'' in \emph{2011 IEEE International Conference on Smart Grid
  Communications (SmartGridComm)}.\hskip 1em plus 0.5em minus 0.4em\relax IEEE,
  2011, pp. 184--189.

\bibitem{kim2013topology}
J.~Kim and L.~Tong, ``On topology attack of a smart grid: Undetectable attacks
  and countermeasures,'' \emph{IEEE Journal on Selected Areas in
  Communications}, vol.~31, no.~7, pp. 1294--1305, 2013.

\bibitem{deng2017ccpa}
R.~Deng, P.~Zhuang, and H.~Liang, ``Ccpa: Coordinated cyber-physical attacks
  and countermeasures in smart grid,'' \emph{IEEE Transactions on Smart Grid},
  vol.~8, no.~5, pp. 2420--2430, 2017.

\bibitem{soltan2018react}
S.~Soltan, M.~Yannakakis, and G.~Zussman, ``React to cyber attacks on power
  grids,'' \emph{IEEE Transactions on Network Science and Engineering}, vol.~6,
  no.~3, pp. 459--473, 2018.

\bibitem{tate2008line}
J.~E. Tate and T.~J. Overbye, ``Line outage detection using phasor angle
  measurements,'' \emph{IEEE Transactions on Power Systems}, vol.~23, no.~4,
  pp. 1644--1652, 2008.

\bibitem{tate2009double}
------, ``Double line outage detection using phasor angle measurements,'' in
  \emph{2009 IEEE Power \& Energy Society General Meeting}.\hskip 1em plus
  0.5em minus 0.4em\relax IEEE, 2009, pp. 1--5.

\bibitem{Zhu12TPS}
H.~Zhu and G.~B. Giannakis, ``Sparse overcomplete representations for efficient
  identification of power line outages,'' \emph{IEEE Transactions on Power
  Systems}, vol.~27, no.~4, pp. 2215--2224, November 2012.

\bibitem{chen2014efficient}
J.-C. Chen, W.-T. Li, C.-K. Wen, J.-H. Teng, and P.~Ting, ``Efficient
  identification method for power line outages in the smart power grid,''
  \emph{IEEE Transactions on Power Systems}, vol.~29, no.~4, pp. 1788--1800,
  2014.

\bibitem{stevenson1994power}
\BIBentryALTinterwordspacing
W.~Stevenson and J.~Grainger, \emph{Power System Analysis}.\hskip 1em plus
  0.5em minus 0.4em\relax McGraw-Hill Education, 1994. [Online]. Available:
  \url{https://books.google.com/books?id=NBIoAQAAMAAJ}
\BIBentrySTDinterwordspacing

\bibitem{Stott09TPS}
B.~Stott, J.~Jardim, and O.~Alsac, ``{DC} power flow revisited,'' \emph{IEEE
  Transactions on Power Systems}, vol.~24, no.~3, pp. 1290--1300, 2009.

\bibitem{Garcia16TPS}
M.~Garcia, T.~Catanach, S.~Vander~Wiel, R.~Bent, and E.~Lawrence, ``Line outage
  localization using phasor measurement data in transient state,'' \emph{IEEE
  Transactions on Power Systems}, vol.~31, no.~4, pp. 3019--3027, 2016.

\bibitem{dagle2010north}
J.~E. Dagle, ``{The North American SynchroPhasor Initiative (NASPI)},'' in
  \emph{IEEE PES General Meeting}.\hskip 1em plus 0.5em minus 0.4em\relax IEEE,
  2010, pp. 1--3.

\bibitem{PMUdeployment}
``{SynchroPhasor} technology fact sheet,'' North American SynchroPhasor
  Initiative, October 2014,
  \url{https://www.naspi.org/sites/default/files/reference\_documents/33.pdf}.

\bibitem{NASPI2015Kit}
``{NASPI} synchrophasor starter kit (draft),'' North American SynchroPhasor
  Initiative (NASPI), October 2015,
  \url{https://www.naspi.org/sites/default/files/reference\_documents/4.pdf}.

\bibitem{jones2013three}
K.~D. Jones, J.~S. Thorp, and R.~M. Gardner, ``Three-phase linear state
  estimation using phasor measurements,'' in \emph{2013 IEEE Power \& Energy
  Society General Meeting}.\hskip 1em plus 0.5em minus 0.4em\relax IEEE, 2013,
  pp. 1--5.

\bibitem{jones2014methodology}
K.~D. Jones, A.~Pal, and J.~S. Thorp, ``Methodology for performing
  synchrophasor data conditioning and validation,'' \emph{IEEE Transactions on
  Power Systems}, vol.~30, no.~3, pp. 1121--1130, 2014.

\bibitem{WASA}
``{WASA} and the roadmap to {WAMPAC} at {SDG\&E},'' September 2020,
  \url{https://quanta-technology.com/wp-content/uploads/2020/09/WASA-and-the-Roadmap-to-WAMPAC-at-SDGE.pdf}.

\bibitem{WAMPACsecurity}
``Wide area monitoring, protection, and control systems {(WAMPAC)} standards
  for cyber security requirements,'' National Electric Sector Cybersecurity
  Organization Resource (NESCOR), October 2012,
  \url{https://smartgrid.epri.com/doc/ESRFSD.pdf}.

\bibitem{pal2006robust}
B.~Pal and B.~Chaudhuri, \emph{Robust control in power systems}.\hskip 1em plus
  0.5em minus 0.4em\relax Springer Science \& Business Media, 2006.

\bibitem{lu2016under}
M.~Lu, W.~ZainalAbidin, T.~Masri, D.~Lee, and S.~Chen, ``Under-frequency load
  shedding (ufls) schemes-a survey,'' \emph{International Journal of Applied
  Engineering Research}, vol.~11, no.~1, pp. 456--472, 2016.

\bibitem{verifiable21arXiv}
Y.~Huang, T.~He, N.~R. Chaudhuri, and T.~L. Porta, ``Verifiable failure
  localization in smart grid under cyber-physical attacks,'' arXiv: 2101.07129,
  Jan. 2021, \url{https://arxiv.org/abs/2101.07129}.

\bibitem{zimmerman2019matpower}
R.~D. Zimmerman and C.~E. Murillo-S{\'a}nchez, ``Matpower 7.0 user’s
  manual,'' \emph{Power Systems Engineering Research Center}, vol.~9, 2019.

\bibitem{dasgupta2008algorithms}
S.~Dasgupta, C.~H. Papadimitriou, and U.~V. Vazirani, \emph{Algorithms}.\hskip
  1em plus 0.5em minus 0.4em\relax McGraw-Hill Higher Education New York, 2008.

\bibitem{mangasarian1994nonlinear}
O.~L. Mangasarian, \emph{Nonlinear programming}.\hskip 1em plus 0.5em minus
  0.4em\relax SIAM, 1994.

\end{thebibliography}

\newpage
\section*{{Appendix: Additional Proofs}}

\begin{proof}[Lemma~\ref{lem:supp in V_H}]
{For a link $(s,t)$, define a column vector $\bfx_{st}\in \{-1,0,1\}^{|V| }$, which has $1$ in $s$-th element, $-1$ in $t$-th element, and $0$ elsewhere. The failure of links in $F$ changes the admittance matrix by\footnote{There was a mistake in the proof of \cite[Lemma~1]{Soltan18TCNS}, which claimed that $\bfB' = \bfB - \sum_{(s,t)\in F}b_{st} \bfx_{st} \bfx_{st}^T$.}
\begin{align}\label{eq:B'}
\bfB' = \bfB + \sum_{(s,t)\in F}b_{st} \bfx_{st} \bfx_{st}^T,
\end{align}
where $b_{st}$ is the $(s,t)$-th element in $\bfB$.
Before the attack, we have $\bfB\bftheta = \bfP$.  After the attack, we have $\bfB'\bftheta' = \bfP'  = \bfP - \Delta$. Therefore, the following holds:
\begin{align}
&\bfB\bftheta - \bfB'\bftheta' = \bfDelta \label{eq:supp proof 1}\\
&\Rightarrow \bfB(\bftheta - \bftheta') - \bfDelta = \sum_{(s,t)\in F}b_{st} \bfx_{st} \bfx_{st}^T \bftheta' \label{eq:supp proof}\\
&\Rightarrow \supp(\bfB(\bftheta-\bftheta')-\bfDelta) \subseteq \bigcup_{(s,t)\in F}\{s,t\} \subseteq V_H,
\end{align}
where (\ref{eq:supp proof}) is obtained by plugging in (\ref{eq:B'}) into (\ref{eq:supp proof 1}). }
\end{proof}

\begin{proof}[Theorem~\ref{thm:recovery of phase angles}]
{By Lemma~\ref{lem:supp in V_H}, we see that $\bfB_{\bar{H}|G}(\bftheta - \bftheta') - \bfDelta_{\bar{H}} = \mathbf{0}$. Writing this equation in more detail shows that
\begin{align}
&\bfB_{\bar{H}|H}(\bftheta_H - \bftheta'_H) + \bfB_{\bar{H}|\bar{H}}(\bftheta_{\bar{H}} - \bftheta'_{\bar{H}}) - \bfDelta_{\bar{H}} = \mathbf{0} \\
&\Rightarrow \bfB_{\bar{H}|H}\bftheta'_H = \bfB_{\bar{H}|H}\bftheta_H + \bfB_{\bar{H}|\bar{H}}(\bftheta_{\bar{H}}-\bftheta'_{\bar{H}}) - \bfDelta_{\bar{H}}. \label{eq:phase angle proof}
\end{align}
Since both $\bfB_{\bar{H}|H}$ and the righthand side of (\ref{eq:phase angle proof}) are known to the control center, we can uniquely recover $\bftheta'_H$ if $\bfB_{\bar{H}|H}$ has a full column rank. }
\end{proof}

\begin{proof}[Lemma~\ref{lem:recover delta}]
{As failures can only occur within $E_H$, nodes in $N(v;\bar{H})$ must be in the same island as $v$ after the attack. Under the proportional load shedding policy, we know that (i) if $\exists u\in N(v;\bar{H})$ of the same type as $v$, then we can recover the post-attack active power at $v$ by $p'_v = p_v p'_u/p_u$ and thus recover $\Delta_v$; (ii) if $\exists u\in N(v;\bar{H})$ of a different type from $v$ (e.g., $u$ is a generator bus but $v$ is a load bus) and $\Delta_u\neq 0$, then $\Delta_v$ must be zero. This proves the claim. }
\end{proof}

\begin{proof}[Lemma~\ref{lem:existence of solution with supp = F}]
{Note that by definition, $\bfx_{st}$ defined in the proof of Lemma~\ref{lem:supp in V_H} is the same as the column corresponding to link $(s,t)$ in $\bfD$. Define a vector $\bfy\in \mathbb{R}^{|E| }$ by
\begin{align}
    y_e = \left\{\begin{array}{ll}
    b_{st}(\theta'_s - \theta'_t) & \mbox{if } e = (s,t) \in F,\\
    0 & \mbox{o.w.}
    \end{array}\right.
\end{align}
Then it is easy to see that $\sum_{(s,t)\in F}b_{st} \bfx_{st} \bfx_{st}^T \bftheta' = \bfD \bfy$.
By (\ref{eq:supp proof}), we have $\bfB(\bftheta - \bftheta') - \bfDelta = \bfD \bfy$. Considering only the equations corresponding to $V_H$ yields
\begin{align}
    \bfB_{H|G}(\bftheta - \bftheta') -\bfDelta_H = \bfD_H \bfy_H,
\end{align}
where we have used the fact that $\bfy_{\bar{H}}= \mathbf{0}$. Thus $\bfx = \bfy_H$ satisfies the conditions in the lemma. }
\end{proof}

\begin{proof}[Theorem~\ref{thm:localize failed links}]
{Condition (1) is implied by \cite[Lemma~3]{Soltan18TCNS}, which proved that $\bfD_H$ has a full column rank if and only if $H$ is acyclic. This combined with Lemma~\ref{lem:existence of solution with supp = F} shows that if $H$ is acyclic, then  (\ref{eq:D_H x equation}) only has one solution, and hence the support of this solution must be $F$.\\
\indent Condition (2) is implied by the proof of \cite[Theorem~2]{Soltan18TCNS}, which showed that if $H$ satisfies this condition, then any solution $\bfx$ to (\ref{eq:D_H x equation}) satisfies $\|\bfx\|_1\geq \|\bfx^*\|_1$, where $\bfx^*$ is a vector satisfying the conditions in Lemma~\ref{lem:existence of solution with supp = F}. Moreover, it showed that $\|\bfx\|_1 = \|\bfx^*\|_1$ only if $\bfx = \bfx^*$. Thus, $\bfx^*$, whose support equals $F$, can be computed by minimizing $\|\bfx\|_1$ s.t. (\ref{eq:D_H x equation}). \looseness=-1}
\end{proof}

\begin{proof}[Lemma~\ref{lem:complexity_p0}]
{We will prove the claim by a reduction from the \emph{subset sum problem}, which is known to be NP-hard \cite{dasgupta2008algorithms}. Given any set of non-negative integers $\{f_i\ge 0\}_{i=1}^n$ and a target value $T$, the subset sum problem determines whether there exists $\{x_i\in \{0,1\}\}_{i=1}^n$ such that $\sum_{i=1}^n f_i x_i = T$. For each subset sum instance, we construct the following star-shaped attacked area $H$: let $H=(V_H,E_H)$ such that $V_H$ is composed of $n+1$ nodes, where node $u_0$ is the hub with $p_{u_0} = 0$ and $\theta_{u_0}' = 0$, and node $u_i$ ($i\in [n]$ for $[n]:=\{1,\ldots, n\}$) is incident to only one link $e_i = (u_0,u_i)$, with $p_{u_i} = -f_i$, $\theta_{u_i}' = -f_i$, and $r_{e_i} = 1$. In addition, $u_0$ is connected to $v\in V_{\bar{H}}$, with $\theta_{v}'=\sum_{i=1}^n f_i-T$, through link $e_0=(u_0,v)$ with $r_{e_0}=1$.
\\ \indent
By substituting \eqref{eq:pf_constraint} and $p_{u_0} = 0$, \eqref{eq:const_valid_load} for node $u_0$ becomes $\tilde{\bfD}_{H,u_0}\bfx_H = \bfB_{u_0|G}\bftheta'$,
where $\tilde{\bfD}_{H,u_0}$ is the row of $\tilde{\bfD}_{H}$ corresponding to node $u_0$. Since $(\tilde{D}_{H,u_0})_i = \frac{\theta_{u_0}'-\theta_{u_i}'}{r_{e_i}}$, it is easy to check that $\tilde{\bfD}_{H,u_0}\bfx_H = \sum_{i=1}^n f_i x_i$. Moreover, $\bfB_{u_0|G}\bftheta' = \sum_{i=1}^n f_i + \frac{\theta_{u_0}'-\theta_{v}'}{r_{e_0}} = T$. Since $u_i$ ($i\in [n]$) is connected to only one link $e_i = (u_0,u_i)$, we have that $\tilde{\bfD}_{H,u_i}\bfx_H = -f_ix_i$ and $\bfB_{u_i|G}\bftheta' = -f_i$. Thus, \eqref{eq:const_valid_load} for $u_i$ becomes $-f_i \le -f_ix_i \le 0$, which is satisfied whatever value $x_i$ takes. Therefore, a subset sum instance returns true if and only if the instance of (P0) constructed as above is feasible, which completes the proof.}
\end{proof}

\begin{proof}[Lemma~\ref{lem:ground_alter_gale}]
{We prove the lemma in two steps. First, note that $\bm{c}^* = \bm{0}$ corresponding to the ground-truth $F$ is feasible for \eqref{eq:primal_theo_both}. If $\hat{F}$ is returned by Algorithm~\ref{alg: fedgeDet} with $e \in Q_m$, there must exist a corresponding optimal solution $\bm{c}$ to \eqref{eq:primal_theo_both} with $c_e \le \eta-1$ and $\bm{1}^T\bm{c} \le \bm{1}^T\bm{c}^* = 0$. Together with the feasibility constraints in \eqref{eq:primal_theo_both}, $\vc$ must satisfy
\begin{align}\label{eq:feasibility_incorrect}
    [\bm{A}_D^T, \bm{A}_x^T, \bm{W}^T, \bm{1}]^T\bm{c} \le [\bm{g}_D^T, \bm{g}_x^T, \bm{g}_w^T, 0]^T,
\end{align}
where $\bm{W}$ and $\bm{g}_w$ are defined such that $e \in Q_m$. To prove $e\notin Q_m$, we only need to show the infeasibility of \eqref{eq:feasibility_incorrect}, which can be proved if there is no solution to \eqref{eq:feasibility_incorrect} when $\bm{W}$ and $\bm{g}_w$ are defined for $Q_m = \{e\}, Q_f = \emptyset$. This is because a linear system must be infeasible if there is no solution to a subset of its inequalities. According to Gale's theorem of alternative~\cite{mangasarian1994nonlinear}, there is no solution to \eqref{eq:feasibility_incorrect} if and only if there exists solutions $\bm{z}\ge \bm{0}$ to \eqref{eq:alter_gale}, which completes the proof.
}
\end{proof}

\begin{proof}[Theorem~\ref{lem: final_lemma}]
{We first prove the condition for a failed link $e\in F$. Based on Lemma~\ref{lem:ground_alter_gale}, we prove by constructing a solution to \eqref{eq:alter_gale} w.r.t $Q_f = \emptyset$ and $Q_m = \{e\}$. We prove by directly constructing the following $\bm{z}$ for \eqref{eq:alter_gale}: (\romannumeral1) $z_* = f_{T,0}$; (\romannumeral2) {$\forall v\in V_H$, set $z_{D,v} = \sum_{U_i\in T_n}\mathbb{I}_{U_i}(v) R_{U_i}$, $z_{D,-v} = \sum_{U_i\in T_p}\mathbb{I}_{U_i}(v) R_{U_i}$, where $\mathbb{I}_{U_i}(v)$ is the indicator function whose value is 1 if $v \in U_i$ and 0 otherwise;} (\romannumeral3) $\forall l\in E_H\setminus F$, set $z_{x-,l} = \tilde{D}_{T,l}+z_*$; (\romannumeral4) $\forall l\in F$ and $l\ne e$, set $z_{x+,l} = |\tilde{D}_{T,l}|-z_*$; (\romannumeral5) $z_{w,m,e} = |\tilde{D}_{T,e}|-z_*$; (\romannumeral6) the rest entries of $\bm{z}$ are set as $0$. Note that $\forall l\in E_H\setminus F$, $z_{x-,l}\ge 0$ due to the definition of $f_{T,0}$. Furthermore, $\forall l\in F\setminus \{e\}$, $z_{x+,l}\ge 0$ since $|\tilde{D}_{T,l}| \ge f_{T,0}$ by the assumption on $T$, and $z_{w,m,e}\ge 0$ for a similar reason. {We will show that \eqref{eq:alter_gale} is satisfied under this assignment. First, note that according to the definition in \eqref{eq:def_DTe}, $\tilde{D}_{T,l} \leq 0$ for all $l\in F$, which implies that $\tilde{D}_{T,l} < 0$ for all $l\in S_T$ and $\tilde{D}_{T,l} \ge 0, \forall l\in (E_H\setminus F) \setminus S_T$. Thus, $\forall l\in F\setminus\{e\}$, the left-hand-side of \eqref{eq:alter_gale_eq} can be expanded as $\sum_{U_i\in T_n} R_{U_i}\tilde{D}_{U_i,l} + \sum_{U_i\in T_p} R_{U_i}(-\tilde{D}_{U_i,l}) + z_{x+,l} + z_* = \tilde{D}_{T,l} + |\tilde{D}_{T,l}|-z_* + z_* = 0$. Similarly, the row of the left-hand-side of \eqref{eq:alter_gale_eq} corresponding to $e$ can be expanded as $\tilde{D}_{T,e} + z_{w,m,e} + z_* = 0$, while the rows corresponding to $l\in E_H\setminus F$ can be expanded as $\tilde{D}_{T,l} - z_{x-,l} + z_* = 0$. Moreover, the left-hand-side of  \eqref{eq:alter_gale_ineq} can be expanded as $f_{T,g} + (\eta-1)(|\tilde{D}_{T,e}| - f_{T,0})$, which satisfies \eqref{eq:alter_gale_ineq} due to \eqref{eq: no_miss_hyper_f}. Note that this assignment of $\bm{z}$ is valid for any possible $Q_m$ and $Q_f$ with $e\in Q_m$.}  That is to say, there is always a non-negative solution to \eqref{eq:alter_gale} if $e\in Q_m$, which implies that $e$ will not be missed by Algorithm~\ref{alg: fedgeDet} according to Lemma~\ref{lem:ground_alter_gale}.
\\ \indent
Next, we prove the condition for an operational link $e'\in E_H\setminus F$. Again, we prove by constructing a solution to \eqref{eq:alter_gale} w.r.t $Q_f = \{e'\}$ and $Q_m = \emptyset$. To this end, we construct the following assignment for $\bm{z}$: (\romannumeral1) $z_* = f_{T,1}$; (\romannumeral2) {$\forall v\in V_H$, set $z_{D,v} = \sum_{U_i\in T_n}\mathbb{I}_{U_i}(v) R_{U_i}$, $z_{D,-v} = \sum_{U_i\in T_p}\mathbb{I}_{U_i}(v) R_{U_i}$}; (\romannumeral3) $\forall l\in E_H\setminus F$ and $l\ne e'$, set $z_{x-,l} = \tilde{D}_{T,l}+z_*$; (\romannumeral4) $\forall l\in F$, set $z_{x+,l} = |\tilde{D}_{T,l}|-z_*$; (\romannumeral5) $z_{w,f,e'} = \tilde{D}_{T,e'}+z_*$; (\romannumeral6) the rest entries of $\bm{z}$ are set as $0$. We will show that $\bm{z} \ge 0$. For $l\in F$, $z_{x+,l}\ge 0$ due to the definition of $f_{T,1}$.  Thus, $\forall l\in E_H\setminus(F\cup\{e'\})$, if $l\notin S_T$, then $z_{x-,l}\ge 0$ since $\tilde{D}_{T,l} \ge 0$; if $l\in S_T$, then $z_{x-,l}\ge f_{T,1}- f_{T,0} \ge 0$. Similarly, $z_{w,f,e'}\ge 0$. Furthermore, the left-hand-side of \eqref{eq:alter_gale_ineq} can be expanded as $[\bm{g}_D^T, \bm{g}_x^T, \bm{g}_w^T, \bm{0}]\bm{z} = f_{T,g} - \eta z_{w,f,e'}$, where $z_{w,f,e'} = z_* - |\tilde{D}_{T,e'}|$ if $e'\in S_T$ and $z_{w,f,e'} = z_* + |\tilde{D}_{T,e'}|$ if $e'\notin S_T$. Then, it is easy to check that $\eqref{eq:alter_gale}$ is satisfied under this assignment, which completes the proof according to Lemma~\ref{lem:ground_alter_gale}. }
\end{proof}

\begin{proof}[Corollary~\ref{cor:cond_connected}]
{We only prove the case that $H$ contains no generator bus since the other case can be proved similarly. We first prove that any failed link $l\in F$ will not be missed ($l\in \hat{F}$). Under Assumption~\ref{as:island_allLoad}, 
link $l$ must have one endpoint (say $u$) such that $\tilde{D}_{u,l} < 0$. Next, we will build a hyper-node $U$ such that the induced subgraph is a tree rooted at node $u$. Specifically, such hyper-node can be constructed by breadth-first search (BFS) starting from node $u$. In the first iteration of BFS, we start with $U = \{u\}$ and add a neighbor $v_i$ of $u$ into $U$ if $e = (u,v_i) \in F$ with $\tilde{D}_{u,l}\tilde{D}_{u,e} <0$ or $e = (u,v_i) \in E_U\setminus F$ with $\tilde{D}_{u,l}\tilde{D}_{u,e} >0$. Then, we repeatedly add node $v$ into $U$ if $\exists e=(s,v)\in E_U\cap F$ such that $\tilde{D}_{U,l}\tilde{D}_{U,e} <0$ or $\exists e=(s,v)\in E_U\setminus F$ such that $\tilde{D}_{U,l}\tilde{D}_{U,e} >0$. This procedure will terminate since $H$ is acyclic, and the constructed $U$ will satisfy condition~1) and condition~2) of Theorem~\ref{lem:no_Miss_hyper_node}. Since all nodes $u \in U$ are load buses, $\tilde{D}_{U,l} < 0$, and the grid stays connected after failure, we have $f_{U,g} = -\sum_{u\in U}\Delta_u = 0$, which satisfies condition~3) of Theorem~\ref{lem:no_Miss_hyper_node}. Thus, we have $F\subseteq \hat{F}$.
\\ \indent
Next, we show that any operational link $e\in E_H\setminus F$ will not be falsely detected by Algorithm~\ref{alg: fedgeDet} ($e\notin \hat{F}$). Under Assumption~\ref{as:island_allLoad}, link $e$ must have have one endpoint (say $u$) such that $\tilde{D}_{u,e} > 0$. The hyper-node $U$ can be constructed as follows: start with $U = \{ u \}$, add node $v$ into $U$ if $\exists e'= (s,v)\in E_U\cap F$ or $\exists e' = (s,v)\in E_U\setminus F$ such that $\tilde{D}_{U,e}\tilde{D}_{U,e'} < 0$. The resulting hyper-node must satisfy condition~1) and condition~2) of Theorem~\ref{lem:no_fa_hyper_direc}. Again, we have $f_{U,g} = -\sum_{u\in U}\Delta_u = 0$, which leads to satisfaction of condition~3) in Theorem~\ref{lem:no_fa_hyper_direc}. Therefore, we have $\hat{F} \subseteq F$.}
\end{proof}

\begin{proof}[Corollary~\ref{col:no_gen_load}]
{First, we construct a set of fail-cover hyper-nodes $T$ as required in Theorem~\ref{lem: final_lemma}. Let $V_{\mathcal{L},c}$ contain all the nodes in $\bigcup_{H_i\in \mathcal{L}}V_i$ that are incident to at least one link in $E_c$. Formally, $V_{\mathcal{L},c}:= \{v_j\}_{j=1}^{K_1}$, 
where $\forall v_j\in V_{\mathcal{L},c}$, $\exists l\in E_c$ such that $\tilde{D}_{v_j,l} \ne 0$. We construct $T :=\{U_j\}_{j=1}^{K_1}$, where $U_j = \{v_j\}$, as the set of fail-cover hyper-nodes, which automatically satisfies condition~2) in Theorem~\ref{lem: final_lemma}. 
\\ \indent
Next, we will show that the constructed $T$ satisfies condition~1) in Theorem~\ref{lem: final_lemma} with strict inequality. To this end, consider any $e=(u_1,u_2)\in S_T$. Suppose that $u_1,u_2\in V_i$. Recall from the proof of Theorem~\ref{lem: final_lemma} that $\tilde{D}_{T,e}<0$, and hence at least one of $U_1:=\{u_1\}$ and $U_2:=\{u_2\}$ must be in $T$ (as otherwise $\tilde{D}_{T,e}=0$). If only $U_1\in T$ ($U_2\not\in T$), then by \eqref{eq:def_DTe}, $\tilde{D}_{T,e}=R_{U_1}\tilde{D}_{U_1,e}$ if $\tilde{D}_{V_i,l}<0$, $\forall l\in E_{c,i}$, and $\tilde{D}_{T,e}=-R_{U_1}\tilde{D}_{U_1,e}$ if $\tilde{D}_{V_i,l}>0$, $\forall l\in E_{c,i}$. To satisfy $\tilde{D}_{T,e}<0$, we must have $\tilde{D}_{U_1,e} \tilde{D}_{V_i,l}>0$ ($\forall l\in E_{c,i}$), and thus $e\in S_{U_1}$. Therefore,
\begin{align}
    |\tilde{D}_{T,e}| &= R_{U_1}|\tilde{D}_{U_1,e}| \leq R_{U_1}f_{U_1,0} < R_{U_1}f_{U_1,1}, \label{eq:proof, U1 in T}
\end{align}
where $|\tilde{D}_{U_1,e}| \leq f_{U_1,0}$ is because of the definition of $f_{U_1,0}$ and that $e\in S_{U_1}$, and $f_{U_1,0} < f_{U_1,1}$ by condition 2) in this corollary. 
If $U_1,U_2\in T$, then $e$ must be in one and only one of $S_{U_1}$ and $S_{U_2}$, as $\tilde{D}_{U_1,l_1}\tilde{D}_{U_2,l_2} >0$ for all $l_1\in E_{U_1}\cap E_c$ and $l_2\in E_{U_2}\cap E_c$. Suppose that $e\in S_{U_1}$. By \eqref{eq:def_DTe}, $\tilde{D}_{T,e} = R_{U_1}\tilde{D}_{U_1,e}+R_{U_2}\tilde{D}_{U_2,e}$ if $\tilde{D}_{V_i,l}<0$ $\forall l\in E_{c,i}$, and $\tilde{D}_{T,e} = -R_{U_1}\tilde{D}_{U_1,e} - R_{U_2}\tilde{D}_{U_2,e}$ if $\tilde{D}_{V_i,l}>0$ $\forall l\in E_{c,i}$. If $R_{U_1}\le R_{U_2}$, we will have $\tilde{D}_{T,e} \ge 0$ since $\tilde{D}_{U_1,e} = -\tilde{D}_{U_2,e}$, which contradicts with $e\in S_T$. Thus, $R_{U_1}> R_{U_2}$, and hence
\begin{align}
    |\tilde{D}_{T,e}| &=(R_{U_1}-R_{U_2}) \cdot |\tilde{D}_{U_1,e}| \nonumber\\
    &\leq R_{U_1} |\tilde{D}_{U_1,e}| < R_{U_1} f_{U_1,1}, \label{eq:proof, U1, U2 in T}
\end{align}
where the last inequality holds for the same reason as \eqref{eq:proof, U1 in T}.
\\ \indent
Then, it suffices to prove that $\max_{U\in T}\{R_{U}f_{U,1}\}\le f_{T,1}$.  To see this, note that $\forall U\in T, R_{U}f_{U,1} = \max_{U'\in T}f_{U',1}$. Thus,
\begin{align}
f_{T,1} &= \min_{l\in E_c} \sum_{U\in T}\mathbb{I}_{E_U}(l)R_U|\tilde{D}_{U,l}| \geq \min_{l\in E_c} \sum_{U\in T}\mathbb{I}_{E_U}(l)R_U f_{U,1} \nonumber \\
&= \min_{l\in E_c}\sum_{U\in T}\mathbb{I}_{E_U}(l)\max_{U'\in T}f_{U',1} \ge \max_{U'\in T}f_{U',1}.
\end{align}
Combined with \eqref{eq:proof, U1 in T} and \eqref{eq:proof, U1, U2 in T}, this leads to $|\tilde{D}_{T,e}| < f_{T,1}$ for any $e\in S_T$. Therefore, we have $f_{T,0} < f_{T,1}$. \\ \indent
Finally, due to assumption~3) of this corollary, we have $f_{U,g} =0, \forall U\in T$, which leads to $f_{T,g} = 0$. Thus, due to  $f_{T,1}>f_{T,0}$, \eqref{eq: no_miss_hyper_f} holds for all $e\in E_c$ and \eqref{eq: no_fa_hyper_f} holds for all $e'\in E_H\setminus E_c$, which proves the corollary by Theorem~\ref{lem: final_lemma}.  }
\end{proof}

\end{document}